\pdfoutput=1
\RequirePackage{ifpdf}
\ifpdf 
\documentclass[pdftex]{sigma}
\else
\documentclass{sigma}
\fi

\usepackage{tikz}
\usetikzlibrary{arrows}
\usetikzlibrary{positioning,decorations.text}

\numberwithin{equation}{section}

\newtheorem{thm}{Theorem}[section]
\newtheorem{prop}[thm]{Proposition}

\newtheorem{RHP}[thm]{Riemann--Hilbert Problem}

\theoremstyle{definition}

\newtheorem{rem}[thm]{Remark}

\def\PVI{P$_{\text{\rm\scriptsize VI}}$}
\def\Tr{\operatorname{Tr}}
\def\Re{\operatorname{Re}}

\def\dd{{\rm d}}

\begin{document}
\allowdisplaybreaks

\renewcommand{\thefootnote}{}

\renewcommand{\PaperNumber}{093}

\FirstPageHeading

\ShortArticleName{A Riemann--Hilbert Approach to the Heun Equation}

\ArticleName{A Riemann--Hilbert Approach to the Heun Equation\footnote{This paper is a~contribution to the Special Issue on Painlev\'e Equations and Applications in Memory of Andrei Kapaev. The full collection is available at \href{https://www.emis.de/journals/SIGMA/Kapaev.html}{https://www.emis.de/journals/SIGMA/Kapaev.html}}}

\Author{Boris DUBROVIN~$^\dag$ and Andrei KAPAEV~$^\ddag$}

\AuthorNameForHeading{B.~Dubrovin and A.~Kapaev}

\Address{$^\dag$~SISSA, Via Bonomea 265, 34136, Trieste, Italy}
\EmailD{\href{mailto:dubrovin@sissa.it}{dubrovin@sissa.it}}

\Address{$^\ddag$~Deceased}

\ArticleDates{Received February 07, 2018, in final form August 15, 2018; Published online September 07, 2018}

\Abstract{We describe the close connection between the linear system for the sixth Painlev\'e equation and the general Heun equation, formulate the Riemann--Hilbert problem for the Heun functions and show how, in the case of reducible monodromy, the Riemann--Hilbert formalism can be used to construct explicit polynomial solutions of the Heun equation.}


\Classification{34M03; 34M05; 34M35; 34M55; 57M50}

\renewcommand{\thefootnote}{\arabic{footnote}}
\setcounter{footnote}{0}

\section{Introduction}

General Heun equation (GHE) \cite{Heun} is the 2nd order linear ODE with four distinct Fuchsian singularities depending on $6$ arbitrary complex parameters. Without loss of generality, three of the singular points can be placed at $0$, $1$ and $\infty$ while the position~$a$ of the fourth singularity remains not fixed. The canonical form of the general Heun equation (GHE) reads
\begin{gather}\label{GHE}
\frac{\dd^2y}{\dd z^2} +\left(\frac{\gamma}{z}+\frac{\kappa}{z-1}+\frac{\epsilon}{z-a}\right)\frac{\dd y}{\dd z}+\frac{\alpha\beta z-q}{z(z-1)(z-a)}y=0,\\
\gamma+\kappa+\epsilon= \alpha+\beta+1,\qquad a\neq0,1,\nonumber
\end{gather}
where the parameters $\alpha$, $\beta$, $\gamma$, $\kappa$, $\epsilon$ determine the characteristic exponents at the singular points,
\begin{alignat*}{3}
& z=0\colon\quad && \{0,1-\gamma\},& \\
& z=1\colon\quad && \{0,1-\kappa\},& \\
& z=a\colon\quad && \{0,1-\epsilon\},& \\
& z=\infty\colon\quad && \{\alpha,\beta\},&
\end{alignat*}
while the remaining {\em accessory} parameter $q$ depends on global monodromy properties of solutions to~(\ref{GHE}).

The general Heun equation together with its confluent and transformed (trigonometric and elliptic) counterparts like Mathieu, spheroidal wave, Leitner--Meixner, Lam\'e and Coulomb sphe\-roi\-dal equations finds numerous applications in quantum and high energy physics, general re\-la\-tivity, astrophysics, molecular physics, crystalline materials, 3d wave in atmosphere, Bethe ansatz systems etc., see~\cite{AD,Hor, SchW, SL} for a~comprehensive but not exhaustive list of references.

Importance of the Heun equation (\ref{GHE}) is due to the fact that any Fuchsian second order linear ODE with $4$ singular points can be reduced to (\ref{GHE}) by elementary transformations, while a trigonometric or elliptic change of the independent variable yield linear ODEs with periodic and double periodic coefficients, see \cite{BE, Ron, St}.

The GHE is the classical example of the Fuchsian ODE that does not admit any continuous isomonodromy deformation. To overcome the difficulty in the construction of the isomonodromy problem, R.~Fuchs \cite{F,F1} added to four conventional Fuchsian singularities one {\em apparent} singularity that is presented in the equation but is absent in the solution. Position $y$ of this apparent singularity is changed together with the position of the Fuchsian singularity~$x$ according to the second order nonlinear ODE known now as the sixth Painlev\'e equation~\PVI. In~\cite{FIKN}, it was observed that under certain assumptions the apparent singularity disappears at the critical values and movable poles of~$y$, and the linear Fuchsian ODE turns into the general Heun equation.

In the present paper, instead of the scalar second order differential equation with apparent fifth singular point we shall use its first order $2\times2$ matrix version with four Fuchsian singular points. We will give a~detailed proof that the general Heun equation appears at the poles of~\PVI, formulate the Riemann--Hilbert (RH) problem for the general Heun equation and explore some implications of the RH problem scheme to the Heun transcendents. In fact, we show how one can obtain the Heun polynomials (i.e., polynomial solutions of~(\ref{GHE})) within the suggested Riemann--Hilbert formalism.

\section[Reduction of the linear differential system for P$_{\text{\rm\scriptsize VI}}$ to the general Heun equation (GHE)]{Reduction of the linear differential system for P$\boldsymbol{_{\text{\rm\bf \scriptsize VI}}}$\\ to the general Heun equation (GHE)}

\subsection[Isomonodromy deformations of a Fuchsian linear ODE with four singularities]{Isomonodromy deformations of a Fuchsian linear ODE\\ with four singularities}

The modern theory of the isomonodromy deformation was developed in the pioneering work of M.~Jimbo, T.~Miwa, and K.~Ueno \cite{JMU, JM}, although its origin goes back to the classical papers of R.~Fuchs \cite{F}, R.~Garnier \cite{Gar,Gar01}, and L.~Schlesinger \cite{Sch}. We shall briefly outline the theory in the case of the $2\times2$ matrix Fuchsian ODE with four singular points. The reader can find more details in the works \cite{JMU, JM}
(see also Part~1 of the monograph~\cite{FIKN}).

The generic first order $2\times2$ matrix Fuchsian ODE with four singular points can be written in the form
\begin{gather}\label{Fuchsian_p6}
\frac{\dd\Psi}{\dd\lambda}=A(\lambda)\Psi,
\end{gather}
with the coefficient matrix
\begin{gather}\label{Lax_pair_p6}
A(\lambda)= \frac{A_1}{\lambda} +\frac{A_2}{\lambda-x} +\frac{A_3}{\lambda-1},\qquad \sum_{j=1}^3A_j=-\delta\sigma_3,\qquad \delta \neq 0,
\end{gather}
where $\sigma_3=\left(\begin{smallmatrix}1 & 0\\ 0 & -1\end{smallmatrix}\right)$. Below we assume that
\begin{gather*}
\Tr A(\lambda) \equiv 0,
\end{gather*}
and denote by
\begin{gather*}
\pm \alpha_j, \qquad j = 1,2,3, \qquad 2\alpha_j \notin {\mathbb Z}
\end{gather*}
the eigenvalues of the matrix residues $A_j$.

The deformation,
\begin{gather*}
A(\lambda) \equiv A(\lambda, x),
\end{gather*}
with respect to the position of the singularity $\lambda=x$ is isomonodromic iff $\Psi(\lambda) \equiv \Psi(\lambda,x)$ satisfies an auxiliary linear ODE with respect to this variable
\begin{gather*}
\frac{\partial\Psi}{\partial x}=B(\lambda)\Psi,\qquad B(\lambda)= -\frac{A_2}{\lambda-x}.
\end{gather*}
In \cite{JM}, it is shown that the unique zero $y \equiv y(x)$ of the entry $A_{12}(\lambda)$( which is a rational function whose numerator is linear in $\lambda$) satisfies the classical sixth Painlev\'e equation \PVI,
\begin{gather*}
y_{xx}=\frac{1}{2}\left(\frac{1}{y}+\frac{1}{y-1}+\frac{1}{y-x}\right)y_x^2-\left(\frac{1}{x}+\frac{1}{x-1}+\frac{1}{y-x}\right)y_x\nonumber\\
\hphantom{y_{xx}=}{} +\frac{y(y-1)(y-x)}{x^2(x-1)^2}\left\{\alpha_0 +\beta_0\frac{x}{y^2}+\gamma_0\frac{x-1}{(y-1)^2}+\delta_0\frac{x(x-1)}{(y-x)^2}\right\}, 
\end{gather*}
where
\begin{gather*}
\alpha_0 = 2\left(\delta - \frac{1}{2}\right)^2, \qquad \beta_0 = -2\alpha_1^2, \qquad \gamma_0 = 2\alpha_3^2, \qquad \delta_0 = -2\left(\alpha_2^2 - \frac{1}{4}\right).
\end{gather*}
Suitable parameterization of the coefficient matrix $A(\lambda)$ of the Fuchsian equation (\ref{Lax_pair_p6}) and appearance of the sixth Painlev\'e equation \PVI\ satisfied by this zero $y(x)$ are explained in more detail in
Appendix~\ref{A}.

In \cite[p.~86]{FIKN}, it was observed that, at the critical values $y=0,1,x$ and movable poles $y=\infty$, the linear matrix equation (\ref{Lax_pair_p6}) with the parametrization~(\ref{parameterization_p6}) becomes equivalent to the GHE. In the following subsections, we describe the way in which the Heun equation emerges at the movable poles of the Painlev\'e functions with more details than in~\cite{FIKN}.

\subsection[Movable poles of P$_{\text{\rm\scriptsize VI}}$]{Movable poles of P$\boldsymbol{_{\text{\rm\bf \scriptsize VI}}}$}


If $\delta\neq\frac{1}{2}$, then equation \PVI\ admits a 2-parameter family of solutions with the following leading terms \cite{GLSh} of the Laurent expansion
\begin{gather}\label{simple_pole_p6}
y(x)=c_{-1}(x-a)^{-1}+c_0+c_1(x-a) +c_2(x-a)^2 +{\mathcal O}\big((x-a)^3\big),
\end{gather}
where $a\in{\mathbb C}\backslash\{0,1\}$, $c_0\in{\mathbb C}$ are arbitrary, while all other coefficients are determined recursively by $a$, $c_0$, $\sigma=\pm 1$ and the local monodromies $\alpha_j$, $j=1,2,3$,
\begin{gather*}
c_{-1}=\sigma\frac{a(a-1)}{2\big(\delta-\frac{1}{2}\big)},\qquad \sigma\in\{1,-1\},\\
c_1=1-\frac{c_0-\frac{1}{3}}{a}-\frac{c_0-\frac{2}{3}}{a-1}+\frac{\sigma}{3}\left(\delta-\frac{1}{2}\right)\left\{
1-\frac{1}{a}\big(6c_0^2-4c_0+1\big)+\frac{1}{a-1}\big(6c_0^2-8c_0+3\big)\right\}\\
\hphantom{c_1=}{} +\frac{\sigma}{2\big(\delta-\frac{1}{2}\big)}
\left\{1-\frac{2}{3}\left(\alpha_2^2-\frac{1}{4}\right)+\frac{2}{3a}\left(\alpha_3^2-\frac{1}{4}\right)-\frac{2}{3(a-1)}\left(\alpha_1^2-\frac{1}{4}\right)\right\},
\end{gather*}
 etc.

If $\delta=\frac{1}{2}$, then the movable poles of solutions to \PVI\ are double \cite{GLSh},
\begin{gather}\label{double_pole_p6}
y(x)=c_{-2}(x-a)^{-2}+c_{-1}(x-a)^{-1}+c_0+{\mathcal O}(x-a),
\end{gather}
where $a\in{\mathbb C}\backslash\{0,1\}$, $c_{-2}\in{\mathbb C}\backslash\{0\}$ are arbitrary, and all other coefficients are determined recursively by $a$ and $c_{-2}$ and the local monodromies $\alpha_j$, $j=1,2,3$,
\begin{gather*}
c_{-1}=\frac{2a-1}{a(a-1)}c_{-2},\\
c_0=\frac{1}{3}(a+1)\\
\hphantom{c_0=}{} +\frac{c_{-2}}{12a^2(a-1)^2}\left(12a(a-1)+1-4a\alpha_1^2+4(a-1)\alpha_3^2-4a(a-1)\left(\alpha_2^2-\frac{1}{4}\right)\right),
\end{gather*}
etc.

\subsection[The coefficient matrix $A(\lambda)$ at the movable poles of P$_{\text{\rm\scriptsize VI}}$]{The coefficient matrix $\boldsymbol{A(\lambda)}$ at the movable poles of P$\boldsymbol{_{\text{\rm\bf \scriptsize VI}}}$}

In this subsection, our concern is the behavior of $A(\lambda)$ at the poles of~$y(x)$. We shall use the parameterization of $A(\lambda)$ given in (\ref{parameterization_p6}). As we will see, the coefficient matrix is continuous at the simple poles with positive $\sigma$ (that is, $\delta\neq\frac{1}{2}$, $\sigma=+1$) and is singular at the poles of any other kind. In the latter case, the linear ODE can be regularized by a suitable Schlesinger transformation~\cite{JM}, and all three resulting regular linear ODEs are equivalent to the GHE.

\begin{thm}\label{Heun_at_poles} At any pole of a solution to {\rm \PVI}, the associated linear ODE~\eqref{Fuchsian_p6} is equivalent $($in some cases after a suitable regularization$)$ to GHE \eqref{GHE}. Moreover, the pole position becomes the position of the fourth singularity in GHE while the free parameter of the Laurent expansion of the Painlev\'e function determines the accessory parameter in GHE.
\end{thm}
\begin{rem} The part of this statement concerning the relation of the free parameter of the Laurent expansion of the Painlev\'e function and the accessory parameter in GHE, in the case of $\delta \neq \frac{1}{2}$, has been already established in \cite{LLNZ}. In \cite{LLNZ} the authors are using a very different approach based on the discovered in \cite{LLNZ} remarkable connection of the classical conformal blocks and the sixth Painlev\'e equation. The authors of \cite{LLNZ} also make use of the striking fact (first observed in \cite{Slav}) that the Heun equation can be thought of as the quantization of the classical Hamiltonian of~\PVI.
 \end{rem}

In Sections~\ref{2.3.1}--\ref{2.3.5}, we give the detailed proof of Theorem \ref{Heun_at_poles} considering each case individually.

\subsubsection[$\delta\neq 0,\frac{1}{2}$, and $\sigma=+1$ (regular case)]{$\boldsymbol{\delta\neq 0,\frac{1}{2}}$, and $\boldsymbol{\sigma=+1}$ (regular case)}\label{2.3.1}

In this case, the coefficient matrix remains continuous. In more details, using (\ref{simple_pole_p6}) along with~(\ref{diff_system_p6}), one finds
\begin{gather*}
\kappa=\kappa_0(x-a)+{\mathcal O}\big((x-a)^2\big),\qquad \kappa_0={\rm const},\qquad x\to a
\end{gather*}
(see equation~\eqref{parameterization_p6} in Appendix~\ref{appendixA} for the definition of the functions $\kappa=\kappa(x)$ and $\tilde\kappa=\tilde\kappa(x)$). This zero cancels the simple pole of $y(x)$, so the $(1,2)$-entry of the matrix $A(\lambda)$ remains bounded. Furthermore, the direct computer-aided computation yields the continuity of all other entries as well,
\begin{gather}
A(\lambda)= \frac{a_3\lambda^2+b_3\lambda+c_3} {\lambda(\lambda-1)(\lambda-a)} \sigma_3 +\frac{c_+} {\lambda(\lambda-1)(\lambda-a)} \sigma_+\nonumber\\
\hphantom{A(\lambda)=}{} +\frac{b_-\lambda+c_-} {\lambda(\lambda-1)(\lambda-a)} \sigma_- +{\mathcal O}(x - a),\qquad x\to a. \label{A_at_simple_pole+_p6}
\end{gather}
Here and below $\sigma_+=\left(\begin{smallmatrix}0 & 1\\ 0 & 0\end{smallmatrix}\right)$, $\sigma_-=\left(\begin{smallmatrix}0 & 0\\ 1 & 0\end{smallmatrix}\right)$. The coefficients $a_3$, $b_3$, $c_3$, $c_+$, $b_-$, $c_-$ are expressed in terms of the local monodromies $\delta$, $\alpha_1$, $\alpha_2$, $\alpha_3$, the position of the pole $a$ and the coefficient $c_0$ of the Laurent series (\ref{simple_pole_p6}). Complete details can be found in Appendix~\ref{A1}.

\subsubsection[$\delta\neq 0, \frac{1}{2},1$ and $\sigma=-1$ (generic singular case)]{$\boldsymbol{\delta\neq 0, \frac{1}{2},1}$ and $\boldsymbol{\sigma=-1}$ (generic singular case)}\label{2.3.2}

Now, the function $\kappa$ develops a pole as $x\to a$,
\begin{gather*}
\kappa=\kappa_0(x-a)^{-1}+{\mathcal O}(1),\qquad \kappa_0={\rm const},\qquad x\to a,
\end{gather*}
and the matrix $A(\lambda)$ becomes singular. To overcome this difficulty, one can apply a~suitable Schlesinger transformation~\cite{JM}
\begin{gather*}
A(\lambda)\mapsto R(\lambda) A(\lambda) R(\lambda)^{-1}+R_\lambda(\lambda) R(\lambda)^{-1},
\end{gather*}
that regularizes the equation (\ref{Lax_pair_p6}) as $x\to a$.

Assume first that $\delta\neq1$. The needed Schlesinger transformation shifts the formal monodromy at infinity $\delta$ by $-1$, i.e., $\delta\mapsto\delta-1$; it is given explicitly by
\begin{gather*}
R_0(\lambda)= \begin{pmatrix}
\lambda+\dfrac{1+x-2p-y(2\delta+1)}{2(\delta-1)}&
-\dfrac{\kappa}{2\delta-1}\vspace{1mm}\\
\dfrac{2\delta-1}{\kappa}&0
\end{pmatrix},\qquad \delta\neq\frac{1}{2},1.
\end{gather*}

It is straightforward to check that the transformed matrix $\hat A$ remains regular at the pole $x=a$, and the limiting coefficient matrix is as follows
\begin{gather}
\hat A=\big(R_0AR_0^{-1}+(R_0)_{\lambda}R_0^{-1}\big)\bigr|_{x=a}\nonumber\\
\hphantom{\hat A}{} =\frac{\hat a_3\lambda^2+\hat b_3\lambda+\hat c_3} {\lambda(\lambda-1)(\lambda-a)}\sigma_3+\frac{\hat b_+\lambda+\hat c_+}{\lambda(\lambda-1)(\lambda-a)}\sigma_++\frac{\hat c_-}
{\lambda(\lambda-1)(\lambda-a)}\sigma_-,\label{A_regularized_p_at_simple_p6}
\end{gather}
where expressions for the constant parameters $\hat a_3$, $\hat b_3$, $\hat c_3$, $\hat b_+$, $\hat c_+$, $\hat c_-$ in terms of the parameters~$a$,~$c_0$ and the local monodromies can be found in Appendix~\ref{A2}.

\subsubsection[$\delta=1$, $\sigma=-1$ (the first special singular case)]{$\boldsymbol{\delta=1}$, $\boldsymbol{\sigma=-1}$ (the first special singular case)} \label{2.3.3}

Let us proceed to the case $\delta=1$ and $\sigma=-1$. We choose the Schlesinger transformation that changes the formal monodromy at infinity and at the origin by one half
\begin{gather*}
\delta=1\mapsto\delta-\frac{1}{2}=\frac{1}{2},\qquad \alpha_1\mapsto\alpha_1-\frac{1}{2},
\end{gather*}
and is given by the gauge matrix $R_1(\lambda)$,
\begin{gather*}
R_1(\lambda)= \frac{1}{\sqrt{\lambda}} \begin{pmatrix} \lambda+g&-\kappa\\ -\dfrac{g}{\kappa}&1 \end{pmatrix}, \qquad g=-p-y+\frac{z+\alpha_1x}{y}.
\end{gather*}

The transformed coefficient matrix $\check A$ is regular at the pole $x=a$, and its limiting value is as follows
\begin{gather}
\check A=\big(R_1AR_1^{-1}+{R_1}_{\lambda}R_1^{-1}\big) \bigr|_{x=a}\nonumber\\
\hphantom{\check A}{} =\frac{\check a_3\lambda^2 +\check b_3\lambda+\check c_3}{\lambda(\lambda-1)(\lambda-a)}\sigma_3+\frac{\check b_+\lambda+\check c_+}{\lambda(\lambda-1)(\lambda-a)}\sigma_++\frac{\check b_-}
{(\lambda-1)(\lambda-a)}\sigma_-,\label{A_regularized_p_at_simple_delta=1_p6}
\end{gather}
where the explicit expressions for the constant coefficients are given in Appendix~\ref{A3}.

\subsubsection[$\delta=\frac{1}{2}$]{$\boldsymbol{\delta=\frac{1}{2}}$} \label{2.3.4}

If $\delta=\frac{1}{2}$ then the Laurent expansion with the double pole (\ref{double_pole_p6}) implies that the coefficient matrix $A(\lambda)$ is singular at $x=a$. The chosen regularizing Schlesinger transformation shifts $\delta\mapsto\delta+1$,
\begin{gather*}
R_2(\lambda)=\begin{pmatrix}
0&-\dfrac{2}{\tilde\kappa}\\
\dfrac{\tilde\kappa}{2}&\lambda+g_2
\end{pmatrix}, \qquad g_2=-\frac{1}{3}(2p+2y-2\tilde y+x+1).
\end{gather*}

The transformed matrix, $\tilde{A} = R_2AR_2^{-1}+{R_2}_{\lambda}R_2^{-1}$, is regular at the pole $x=a$ and, at this point, takes the value
\begin{gather}\label{A_regularized_double_p6}
\tilde A(\lambda) =\frac{\tilde a_3\lambda^2 +\tilde b_3\lambda+\tilde c_3} {\lambda(\lambda-1)(\lambda-a)}\sigma_3 +\frac{\tilde c_+} {\lambda(\lambda-1)(\lambda-a)}\sigma_+ +\frac{\tilde b_-\lambda+\tilde c_-} {\lambda(\lambda-1)(\lambda-a)}\sigma_-,
\end{gather}
where the coefficients $\tilde{a}_3$, $\tilde{b}_3$, $\tilde{c}_3$, $\tilde{c}_+$, $\tilde{b}_-$, $\tilde{c}_-$ are presented in Appendix~\ref{A4}.

\subsubsection[GHE from the linear ODEs at the poles of P$_{\text{\rm\scriptsize VI}}$]{GHE from the linear ODEs at the poles of P$\boldsymbol{_{\text{\rm\bf \scriptsize VI}}}$}\label{2.3.5}

In this subsection, we show that all the linear matrix ODEs corresponding to the poles of the sixth Painlev\'e function are equivalent to the GHE.

Observe that the coefficient matrices of the form~(\ref{A_at_simple_pole+_p6}) corresponding to $\delta\neq0,\frac{1}{2},1$ and $\sigma=+1$ as well as the regularized coefficient matrix~(\ref{A_regularized_double_p6}) corresponding to $\delta=\frac{1}{2}$ coincide with each other modulo notations. Similarly, mutatis mutandis, the matrix~(\ref{A_regularized_p_at_simple_p6}) for $\delta\neq0,\frac{1}{2},1$, $\sigma=-1$ is the $\sigma_1$-conjugate of the previous coefficient matrices. The coefficient matrix~(\ref{A_regularized_p_at_simple_delta=1_p6}) for $\delta=1$, $\sigma=-1$, is an inessential modification of~(\ref{A_regularized_p_at_simple_p6}). It is enough then to consider the cases~(\ref{A_at_simple_pole+_p6}) and~(\ref{A_regularized_p_at_simple_delta=1_p6}).

Consider first the coefficient matrix (\ref{A_at_simple_pole+_p6}).

The first order matrix equation for the function $\Psi(\lambda)$ is always equivalent to the second order Fuchsian ODE for the entry $\Psi_{1*}(\lambda)$ of the first row of the matrix function $\Psi(\lambda)$. However, extra (apparent) singularities might appear in the process of excluding the entry $\Psi_{2*}(\lambda)$. In the case of~(\ref{A_at_simple_pole+_p6}), however, the rational function representing the $12$~-- entry of matrix (\ref{A_at_simple_pole+_p6}) does not have~$\lambda$ in its numerator. Hence, when the entry $\Psi_{2*}(\lambda)$ is excluded from the system, no apparent singularities appear. Therefore, in the case~(\ref{A_at_simple_pole+_p6}), the entry $\Psi_{1*}(\lambda)$ of the first row of the matrix function $\Psi(\lambda)$ satisfies a linear 2nd order Fuchsian ODE with $4$ singular points without any apparent singularity
and therefore is equivalent to GHE. It is, in fact, straightforward to check that the function
\begin{gather}\label{u_+_def}
u(\lambda)=\lambda^{\alpha_1}(\lambda-a)^{\alpha_2}(\lambda-1)^{\alpha_3}\Psi_{1*}(\lambda)
\end{gather}
satisfies the general Heun equation in its canonical form (\ref{GHE})
\begin{gather}
u''+\left(\frac{1-2\alpha_1}{\lambda}+\frac{1-2\alpha_2}{\lambda-a}+\frac{1-2\alpha_3}{\lambda-1}\right)u'+\frac{\mu\lambda+\nu}{\lambda(\lambda-a)(\lambda-1)}u=0,\nonumber\\
\mu= (\alpha_1+\alpha_2+\alpha_3-\delta)(\alpha_1+\alpha_2+\alpha_3+\delta-2),\label{Heun_gen_can}\\
\nu=\alpha_1+\alpha_2-(\alpha_1+\alpha_2)^2+\alpha_3^2-\delta^2+a(\alpha_1+\alpha_3-(\alpha_1+\alpha_3)^2+\alpha_2^2-\delta^2)+b_3(2\delta-1).\nonumber
\end{gather}
Observe that the expression for the accessory parameter $\nu$ in~(\ref{Heun_gen_can}) besides the pole position and the local monodromies, involves the coefficient $b_3$ in the parameterization of the entry~$A_{11}$. Thus, taking into account formula for $b_3$ in~(\ref{A_at_simple_pole+_p6_coeffs}), we see that the accessory parameter $\nu$ is determined by the free coefficient~$c_0$ (or~$c_{-2}$ in the case~(\ref{A_regularized_double_p6}~--
see (\ref{hatA_coeffs_double_p6}))) in the Laurent expansion of the sixth Painlev\'e transcendent.

For the coefficient matrix (\ref{A_regularized_p_at_simple_p6}), corresponding to $\delta \neq 1$, $\sigma=-1$ a similar statement is valid for the entries of the second row of $\Psi(\lambda)$,
\begin{gather*}
v(\lambda)=\lambda^{\alpha_1} (\lambda-a)^{\alpha_2}(\lambda-1)^{\alpha_3}\Psi_{2*}(\lambda),\nonumber\\
v''+\left(\frac{1-2\alpha_1}{\lambda}+\frac{1-2\alpha_2}{\lambda-a}+\frac{1-2\alpha_3}{\lambda-1}\right)v'+\frac{\hat\mu\lambda+\hat\nu}{\lambda(\lambda-a)(\lambda-1)}v=0,\nonumber\\
\hat\mu=(\alpha_1+\alpha_2+\alpha_3-\delta-1)(\alpha_1+\alpha_2+\alpha_3+\delta-1),\nonumber\\
\hat\nu=\alpha_1+\alpha_2-(\alpha_1+\alpha_2)^2+\alpha_3^2-(\delta-1)^2\nonumber\\
\hphantom{\hat\nu=}{} +a\big(\alpha_1+\alpha_3-(\alpha_1+\alpha_3)^2+\alpha_2^2-(\delta-1)^2\big)+\hat b_3(2\delta-1).
\end{gather*}

In the case (\ref{A_regularized_p_at_simple_delta=1_p6}) corresponding to $\delta=1$, $\sigma=-1$, the function
\begin{gather*}
\check v(\lambda)=\lambda^{\alpha_1-\frac{1}{2}}(\lambda-a)^{\alpha_2}(\lambda-1)^{\alpha_3}\Psi_{2*}(\lambda)
\end{gather*}
satisfies the following Heun equation
\begin{gather*}
\check v''+\left( \frac{1-2\alpha_1}{\lambda}+\frac{1-2\alpha_2}{\lambda-a}+\frac{1-2\alpha_3}{\lambda-1}\right)\check v' +\frac{\check\mu\lambda+\check\nu}{\lambda(\lambda-1)(\lambda-a)}\check v=0,\nonumber\\
\check\mu=(\alpha_1+\alpha_2+\alpha_3-2)(\alpha_1+\alpha_2+\alpha_3),\nonumber\\
\check\nu=b_3-\frac{1}{2}(a+1)+\alpha_1+\alpha_2-(\alpha_1+\alpha_2)^2+\alpha_3^2+a\big(\alpha_1+\alpha_3-(\alpha_1+\alpha_3)^2+\alpha_2^2\big). 
\end{gather*}

This completes the proof of Theorem~\ref{Heun_at_poles}.

\section{Riemann--Hilbert problem approach to the Heun equation}

Main result of this section is the formulation of the RH problem for the general Heun functions in the generic case
\begin{gather*}
2\alpha_1,2\alpha_2,2\alpha_3,2\delta\notin{\mathbb Z}.
\end{gather*}
We shall start, following closely references \cite{FIKN, J}, with the standard definition of the monodromy data for Fuchsian system (\ref{Fuchsian_p6}), (\ref{Lax_pair_p6}) and with the related Riemann--Hilbert problem for the sixth Painlev\'e equation.

\subsection{Monodromy data}

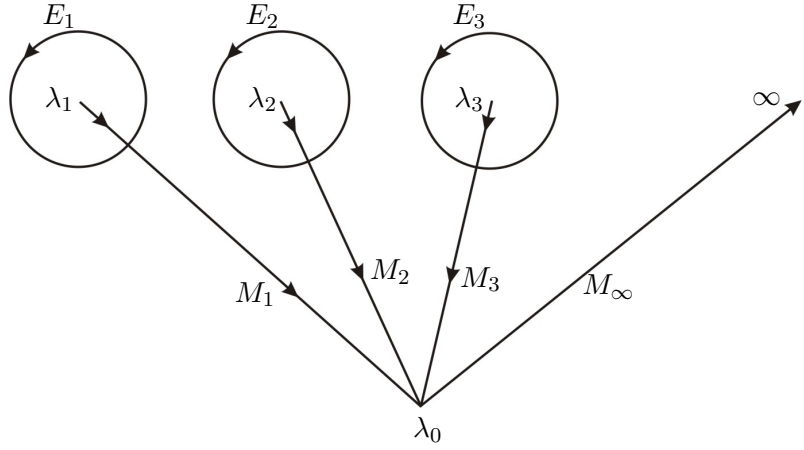
\begin{figure}[t]\centering
\begin{tikzpicture}
\node at (0,0) {\includegraphics[scale=0.9]{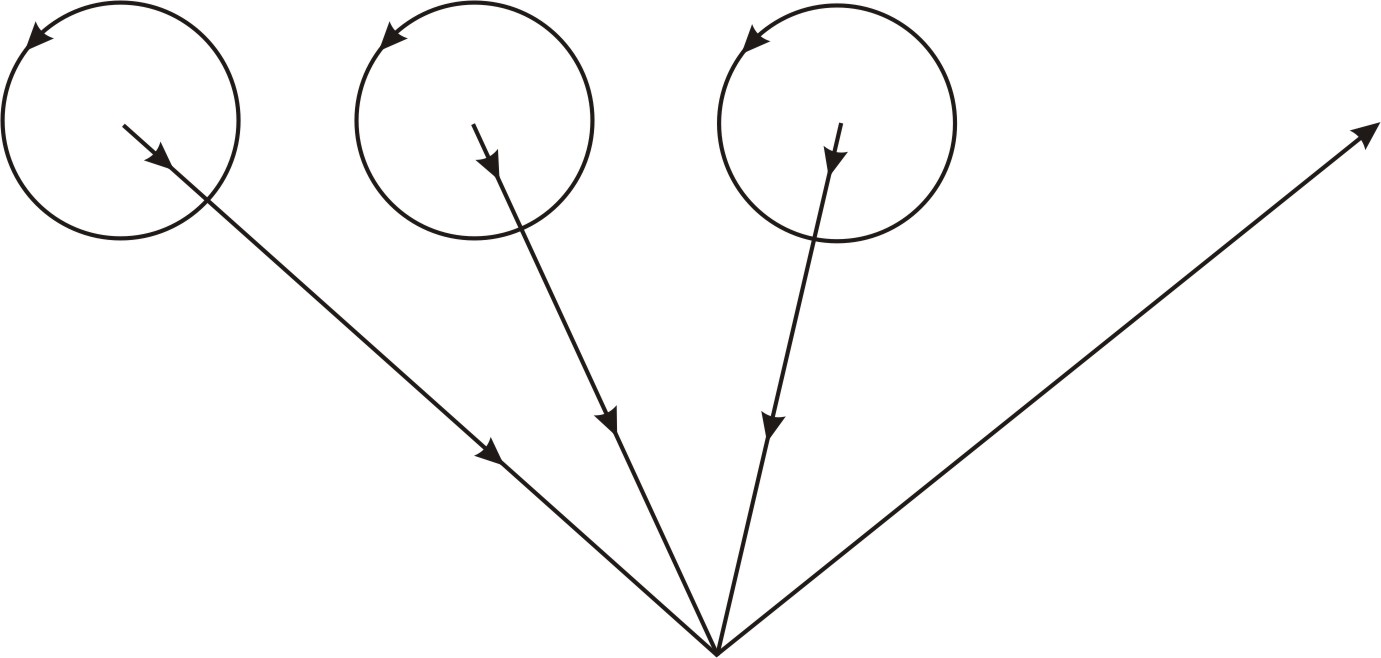}};

\node at (-4.6,1.6) {$\lambda_1$}; \node at (-4.6,2.7) {$E_1$}; \node at (-2.0,-1.0) {$M_1$}; \node at (4.8,1.6) {$\infty$};

\node at (-1.9,1.6) {$\lambda_2$}; \node at (-1.9,2.7) {$E_2$}; \node at (-0.2,-0.7) {$M_2$};

\node at (0.85,1.6) {$\lambda_3$}; \node at (0.85,2.7) {$E_3$}; \node at (1.0,-0.8) {$M_3$}; \node at (2.7,-0.9) {$M_\infty$};

\node at (0.3,-2.8) {$\lambda_0$};


\end{tikzpicture}

\caption{The jump contour $\gamma$ for the RH Problem~\ref{RHP_p6}.}\label{fig1}
\end{figure}

Let $\lambda_1=0$ and $\lambda_3=1$. Then fix a point $\lambda_2=x\in{\mathbb C}\backslash\{0,1,\infty\}$, choose a base point $\lambda_0\in{\mathbb C}\backslash\{0,1,x,\infty\}$ and cut the complex plane along the segments
\begin{gather*}
[\lambda_0,\lambda_1]\cup [\lambda_0,\lambda_2]\cup [\lambda_0,\lambda_3]\cup [\lambda_0,\infty].
\end{gather*}
Encircle the points $\lambda_1=0$, $\lambda_2=x$ and $\lambda_3=1$ using non-intersecting circles ${\mathcal C}_j$, $j = 1,2,3$. Denote $\gamma$ the graph
\begin{gather*}
\gamma = [\lambda_0,\lambda_1]\cup [\lambda_0,\lambda_2]\cup [\lambda_0,\lambda_3]\cup [\lambda_0,\infty]\cup \cup_{j=1}^3{\mathcal{C}}_j
\end{gather*}
and orient it as in Fig.~\ref{fig1}. Denote also $D_j$, $j = 1,2,3$ the interiors of the circles ${\mathcal C}_j$ and $D_{\infty}$ the domain
\begin{gather*}
D_{\infty} = \bigl({\mathbb C}\backslash\gamma\bigr) \backslash\bigl(\cup_{j=1}^3D_j\bigr)
\end{gather*}
The domains $D_j$ are assigned to the principal branches of the Frobenius (canonical) solutions to~(\ref{Fuchsian_p6}) at the Fuchsian singular points defined by the conditions
\begin{gather}
\Psi_j(\lambda)= T_j(I+{\mathcal O}(\lambda-\lambda_j))(\lambda-\lambda_j)^{\alpha_j\sigma_3},\nonumber\\ \lambda \to \lambda_j, \qquad \lambda \in D_j, \qquad j =1,2,3,\qquad \det T_j=1,\nonumber\\
\Psi_{\infty}=\big(I+{\mathcal O}\big(\lambda^{-1}\big)\big)\lambda^{-\delta\sigma_3},\qquad \lambda \to \infty,\qquad \lambda \in D_{\infty}.\label{Frobenius_sol_def}
\end{gather}
Here, the branches of $(\lambda-\lambda_j)^{\alpha_j}$, $j=1,2,3$, and $\lambda^{-\delta}$ are fixed by the condition
\begin{gather*}
\arg(\lambda-\lambda_j)\to\pi,\qquad \arg\lambda\to\pi,\qquad \mbox{as} \quad \lambda\to-\infty.
\end{gather*}

Given a pair of the characteristic exponents $(\alpha_j,-\alpha_j)$, the matrix of eigenvectors $T_j\in {\rm SL}(2,{\mathbb C})$ is determined up to a right diagonal factor. In contrast, the Frobenius solution $\Psi_{\infty}(\lambda)$ is normalized and therefore, as soon as the pair $(-\delta,\delta)$ is fixed, it is determined uniquely.

The matrices of the local monodromy are defined as the branch matrices of the Frobenius solutions
\begin{gather*}
\Psi_j\big(\lambda_j+(\lambda-\lambda_j){\rm e}^{2\pi {\rm i}}\big)= \Psi_j(\lambda){\rm e}^{2\pi {\rm i}\alpha_j\sigma_3},\qquad \Psi_{\infty}\big({\rm e}^{2\pi {\rm i}}\lambda\big)= \Psi_{\infty}(\lambda){\rm e}^{-2\pi {\rm i}\delta\sigma_3}.
\end{gather*}
Introduce also the connection matrices between Frobenius solutions at infinity and at the finite singularities
\begin{gather*}
\Psi_j(\lambda)=\Psi_{\infty}(\lambda)E_j.
\end{gather*}
Similar to $\Psi_j(\lambda)$, the connection matrices $E_j$ are determined modulo arbitrary right diagonal factors. In contrast, the {\it monodromy matrices} $M_j$,
\begin{gather}\label{monodromy_matrices}
M_j=E_j{\rm e}^{2\pi {\rm i}\alpha_j\sigma_3}E_j^{-1},
\end{gather}
are determined uniquely. The monodromy matrices are the branching matrices of the solu\-tion~$\Psi_{\infty}(\lambda)$ at the singular points $\lambda_j$, $j = 1,2,3$; namely, one has that
\begin{gather}\label{global_monodromies}
\Psi_{\infty}\big(\lambda_j+(\lambda-\lambda_j){\rm e}^{2\pi {\rm i}}\big)= \Psi_{\infty}(\lambda)M_j,\qquad j = 1,2,3.
\end{gather}
Together with the matrix
\begin{gather}\label{Minfty}
M_\infty := {\rm e}^{-2\pi {\rm i}\delta\sigma_3}
\end{gather}
they generate the {\it monodromy group} of equation (\ref{Fuchsian_p6})
\begin{gather*}
{\mathbb M} = \langle M_1, M_2, M_3, M_\infty\rangle , \qquad M_j \in {\rm SL}(2,{\mathbb C}),
\end{gather*}
and are subject of one (cyclic) constraint
\begin{gather}\label{cyclic}
M_1M_2M_3=M_{\infty}.
\end{gather}

\looseness=-1 Given the local monodromies, each of the monodromy matrices $M_j$ depends on 2 parameters. The total set of the 6 parameters determining the monodromy matrices $M_j$, $j=1,2,3,\infty$, is subject to a system of 3 scalar constraints. Thus the parameter set of the monodromy data involves generically $3$ parameters. One of these parameters corresponds to the constant factor~$\kappa_0$ determining the auxiliary function $\kappa$~-- see (\ref{parameterization_p6}), (\ref{diff_system_p6}). This is a reflection of the possible conjugation of $A(\lambda)$ by a constant diagonal matrix~-- the action which does not affect the zero of $A_{12}(\lambda)$, i.e., the \PVI\ function~$y(x)$. Neglecting this auxiliary parameter, the {\it space of essential monodromy data}, ${\mathcal M}$, is invariant with respect to an overall conjugation by a diagonal matrix and can be identified with an algebraic variety~-- the monodromy surface, of dimension~2 (see below equation~(\ref{Fricke_cubic})). At the same time, the full space of monodromy data, can be represented as
\begin{gather*}
{\frak M}={\mathcal M}\times{\mathbb C},\qquad \dim{\mathcal M}=2.
\end{gather*}

In \cite{J}, M.~Jimbo has proposed a parameterization of the 2-dimensional monodromy sur\-face~${\mathcal M}$ by the trace coordinates invariant with respect to the overall diagonal conjugation. Namely, letting
\begin{gather*}
a_j=\Tr M_j=2\cos(2\pi\alpha_j),\qquad j=1,2,3,\infty,\qquad \alpha_{\infty}=\delta,\nonumber\\
 t_{ij}=\Tr(M_iM_j)=2\cos(2\pi\sigma_{ij}),\qquad i,j=1,2,3,
\end{gather*}
one finds the relation between all these parameters for a 2-dimensional surface called the Fricke cubic
\begin{gather}
t_{12}t_{23}t_{31} +t_{12}^2+t_{23}^2+t_{31}^2-(a_1a_2+a_3a_{\infty})t_{12}-(a_2a_3+a_1a_{\infty})t_{23}-(a_3a_1+a_2a_{\infty})t_{31}\nonumber\\
\qquad {}+a_1^2 +a_2^2+a_3^2+a_{\infty}^2+a_1a_2a_3a_{\infty}-4=0.\label{Fricke_cubic}
\end{gather}
According to \cite{Iwasaki}, apart from the singular points of the surface (\ref{Fricke_cubic}), the monodromy matrices can be written explicitly in terms of the variables $t_{ij}$. Exact formulae can be found in~\cite{J}
and~\cite{Iwasaki}.

Each point of the surface $\mathcal M$ represents an isomonodromic family of equations (\ref{Fuchsian_p6}) which in turns generates a solution $y(x)$ of the Painlev\'e VI equation. Hence, the \PVI\ transcendents can be parameterized by the points of $\mathcal M$. In fact, at generic points of the Fricke cubic one can use any pair of the parameters~$t_{ij}$ or~$\sigma_{ij}$ to parameterize the set of the corresponding Painlev\'e functions. For instance, one can choose,
\begin{gather}\label{tsdef}
t := t_{12} = \Tr(M_1M_2), \qquad s:= t_{1,3}=\Tr(M_1M_3),
\end{gather}
so that we have the parameterization of the \PVI\ functions by the pair $(t,s)$,
\begin{gather}\label{yts}
y\equiv y(x;t,s).
\end{gather}
It also should be mentioned that some of the physically important solutions, e.g., the so-called classical solutions to \PVI, correspond to non-generic points of the monodromy data set and for their parameterization one can use the full monodromy space ${\frak M}$. We refer to~\cite{Guz2} for more detail on this issue.

\subsection[Riemann--Hilbert problem for P$_{\text{\rm\scriptsize VI}}$]{Riemann--Hilbert problem for P$\boldsymbol{_{\text{\rm\bf \scriptsize VI}}}$}

The {\it inverse monodromy problem}, i.e., the problem of reconstruction of the function~$\Psi$, and hence of the corresponding Painlev\'e function $y(x)$, from their monodromy data is formulated as a~Riemann--Hilbert (RH) problem. The direct and inverse monodromy problems associated with the equation \PVI\ were studied by several authors. We mention here the pioneering paper~\cite{J}, and subsequent papers~\cite{Bo, DM, Guz2}.

We shall now formulate precisely the Riemann--Hilbert problem corresponding to the inverse monodromy problem for Fuchsian $2\times 2$ system~(\ref{Fuchsian_p6}).

\begin{RHP}\label{RHP_p6} Given $x$, $\delta$, $\alpha_j$, $2\delta,2\alpha_j\notin{\mathbb Z}$, $j=1,2,3$, the oriented graph~$\gamma$ shown in Fig.~{\rm \ref{fig1}} $($with $\lambda_1 =0$, $\lambda_2 =x$, and $\lambda_3 =1)$ and the jump matrices $M_j$, $E_j$, $j=1,2,3$, all assigned to the branches of $\gamma$ and satisfying the conditions \eqref{monodromy_matrices}, \eqref{Minfty}, \eqref{cyclic}, find a piecewise holomorphic $2\times2$ matrix function $\Psi(\lambda)$ with the following properties:
\begin{enumerate}\itemsep=0pt
\item[$1)$] $\|\Psi(\lambda)\lambda^{\delta\sigma_3}-I\| \leq C|\lambda|^{-1}$ where $C$ is a constant and $\lambda\to\infty$,
\item[$2)$] $\|\Psi(\lambda)(\lambda-\lambda_j)^{-\alpha_j\sigma_3}\|\leq C_j$ as $\lambda\to\lambda_j$, where $C_j$ are some constants, $j=1,2, 3$,
\item[$3)$] $\|\Psi(\lambda)\|\leq C_0$, where $C_0$ is a constant and $\lambda$ approaches any nodal point of the graph~$\gamma$,
\item[$4)$] across the piecewise oriented contour $\gamma$, the discontinuity condition holds,
\begin{gather*}
\Psi_+(\lambda)=\Psi_-(\lambda)G(\lambda),\qquad \lambda\in\gamma,
\end{gather*}
where $\Psi_+(\lambda)$ and $\Psi_-(\lambda)$ are the left and right limits of $\Psi(\lambda)$ as $\lambda$ transversally approaches the contour $\gamma$, and $G(\lambda)$ is the piecewise constant matrix defined on $\gamma$, see Fig.~{\rm \ref{fig1}}.
\end{enumerate}
\end{RHP}
\begin{prop} If a solution to the RH Problem~{\rm \ref{RHP_p6}} exists it is unique.
\end{prop}
\begin{prop} Having the canonical solutions $\Psi_j(\lambda)$ of \eqref{Fuchsian_p6}, the equations
\begin{gather}\label{Psidef}
\Psi(\lambda) = \Psi_{j}(\lambda), \qquad \lambda\in D_j, \qquad j = 1,2,3,\infty,
\end{gather}
define the function which solves the RH Problem~{\rm \ref{RHP_p6}} whose data are determined by the corresponding monodromy data.
\end{prop}
\begin{prop} Conversely, if for given $x$, $\delta$, $\alpha$, and matrices $M_j$, $E_j$ the RH Problem~{\rm \ref{RHP_p6}} is solvable then the function $\Psi(\lambda)$ satisfies the Fuchsian system \eqref{Fuchsian_p6} whose Frobenius solutions are determined by the solution~$\Psi(\lambda)$ of the RH problem according to the equations~\eqref{Psidef} $($read backwards$)$ and whose monodromy data coincide with the given RH data.
\end{prop}

The proofs of these propositions are standard, see, e.g., \cite{FIKN}.

Assuming that the RH Problem~\ref{RHP_p6} is solvable, the Painlev\'e function can be extracted from the asymptotics of its solution at infinity. Indeed, introducing the matrices ${{\mathcal E}_k}= \{\delta_{ik}\delta_{jk}\}_{i,j=1,2}$, $k=1,2$, by straightforward computations we find the asymptotics of $\Psi_{\infty}(\lambda)$,
\begin{gather}\label{Psi_p6_as_at_8}
\Psi_{\infty}(\lambda)= \left(I +\frac{1}{\lambda}\psi_1 +\frac{1}{\lambda^2}\psi_2 +{\mathcal O}\left(\frac{1}{\lambda^3}\right) \right)
{\rm e}^{\frac{1}{\lambda}d_1\sigma_3 +\frac{1}{\lambda^2}(d_{21}{\mathcal E}_1+d_{22}{\mathcal E}_2)}\lambda^{-\delta\sigma_3}, \qquad \lambda\to\infty,
\end{gather}
where the coefficient matrices $\psi_1$ and $\psi_2$ are off-diagonal. The expressions of $\psi_k$, $d_{kl}$ in terms of the coefficients of $A(\lambda)$ (see (\ref{parameterization_p6})) can be found in Appendix~\ref{psi_k_p6}.

Using (\ref{deltaneq1/2_1_Psi_p6_as_at_8}), namely, the $\sigma_+$-components $(\psi_1)_+$ and $(\psi_2)_+$ of $\psi_1$ and $\psi_2$ respectively, and the scalars $d_1$, $d_{21}$, $d_{22}$, we find
\begin{gather}
\kappa=(2\delta-1)(\psi_1)_+,\qquad
p=-\delta(x+1)+2\delta\frac{(\delta-1)(\psi_2)_+}{\big(\delta-\frac{1}{2}\big)(\psi_1)_+}+d_1\frac{\delta+\frac{1}{2}}{\delta-\frac{1}{2}},\nonumber\\
y=x+1-\frac{(\delta-1)(\psi_2)_+}{\big(\delta-\frac{1}{2}\big)(\psi_1)_+}-\frac{d_1}{\delta-\frac{1}{2}},\nonumber\\
z=-d_{21}+d_{22}-2\delta(d_{21}+d_{22})-\delta x\nonumber\\
\hphantom{z=}{} -\left(\frac{\delta(\delta-1)(\psi_2)_+}{\big(\delta-\frac{1}{2}\big)(\psi_1)_+}+\frac{d_1}{2\big(\delta-\frac{1}{2}\big)}-\delta(x+1)\right)
\left(\frac{(\delta-1)(\psi_2)_+}{\big(\delta-\frac{1}{2}\big)(\psi_1)_+}+\frac{d_1}{\delta-\frac{1}{2}}\right).\label{p6_via_RH_sol}
\end{gather}

\subsection{Riemann--Hilbert problem for the Heun function}

Main result of this section states that the RH problem for the Heun function coincides with that for \PVI\ supplemented by the additional condition of triangularity of the sub-leading term of the asymptotic expansion of $\Psi(\lambda)$ as $\lambda\to\infty$.

Again, we consider the non-resonant case $2\alpha_1$, $2\alpha_2$, $2\alpha_3$, $2\delta\notin{\mathbb Z}$.

\subsubsection[Limiting equation (\ref{Lax_pair_p6}) and the $\Psi$-function at the pole $x=a$ of $y(x)$ as $\delta\neq\frac{1}{2}$ and $\sigma=+1$]{Limiting equation (\ref{Lax_pair_p6}) and the $\boldsymbol{\Psi}$-function\\ at the pole $\boldsymbol{x=a}$ of $\boldsymbol{y(x)}$ as $\boldsymbol{\delta\neq\frac{1}{2}}$ and $\boldsymbol{\sigma=+1}$}

Consider the Laurent expansion (\ref{simple_pole_p6}) with $\sigma=+1$ and the corresponding coefficient mat\-rix~(\ref{A_at_simple_pole+_p6}). For $x$ in a punctured neighborhood of the pole $x=a$, there exists an isomonodromy family of $\Psi$-functions. Furthermore, the continuity of the coefficient matrix with respect to~$x$ implies the continuity of the Frobenius solutions~(\ref{Frobenius_sol_def}) at $x=a$ as well.

On the other hand, the form of the coefficient matrix (\ref{A_at_simple_pole+_p6}) with $a_3=-\delta$ implies the following asymptotics of the solution to the linear ODE $\Psi_{\lambda}=A\Psi$,
\begin{gather}\label{Psi_p6_at_pole_sigma+_8}
\Psi(\lambda)=\left(I+\frac{1}{\lambda}\psi_1+\frac{1}{\lambda^2}\psi_2+{\mathcal O}\left(\frac{1}{\lambda^3}\right)\right) {\rm e}^{\frac{1}{\lambda}d_1\sigma_3+\frac{1}{\lambda^2}d_2\sigma_3}\lambda^{-\delta\sigma_3},\qquad
\lambda\to\infty,
\end{gather}
where the main difference from the asymptotic parameters in~(\ref{Psi_p6_as_at_8}), (\ref{deltaneq1/2_1_Psi_p6_as_at_8}) is the lower-triangular structure of the coefficient~$\psi_1$,
\begin{gather*}
\psi_1=-\frac{b_-}{2\delta+1}\sigma_-,\qquad d_1=-b_3+\delta(a+1),\nonumber\\
\psi_2=\frac{c_+}{2(\delta-1)}\sigma_+-\frac{b_-(a+1+2b_3)+c_-(2\delta+1)}{4(\delta+1)\big(\delta+\frac{1}{2}\big)}\sigma_-,\nonumber\\
d_2=\frac{1}{2}\bigl(-b_3(a+1)-c_3+\delta\big(1+a+a^2\big)\bigr).
\end{gather*}
We point out that the lower-triangular structure of $\psi_1$ implies the lower-triangular structure of the ${\mathcal O}\big(\lambda^{-1}\big)$-term in the expansion~\eqref{Psi_p6_at_pole_sigma+_8}.

All other principal analytic properties of the limiting function $\Psi(\lambda)$ including the leading order asymptotics at the singular points and the monodromy properties coincide with those of the function~$\Psi(\lambda)$ at the regular points of the Painlev\'e transcendent located in a sufficiently small neighborhood of the pole $x=a$.

Thus we have shown that the RH problem for the function $\Psi(\lambda)$ at the pole of the sixth Painlev\'e transcendent $y(x)$ as $\delta\neq\frac{1}{2}$ and $\sigma=+1$, and therefore the RH problem for a solution of the general Heun equation~(\ref{Heun_gen_can}), coincides with the RH Problem~\ref{RHP_p6} supplemented by the condition of the lower triangularity of the coefficient $\psi_1$ at infinity:

\begin{RHP}\label{RHP_at_pole_sigma+}Given $x$, $\delta$, $\alpha_j$, $2\delta,2\alpha_j\notin{\mathbb Z}$, $j=1,2,3$, the oriented graph~$\gamma$ shown in Fig.~{\rm \ref{fig1}} $($with $\lambda_1 =0$, $\lambda_2 =a$, and $\lambda_3 =1)$ and the jump matrices $M_j$, $E_j$, $j=1,2,3$, all assigned to the branches of $\gamma$ and satisfying the conditions \eqref{monodromy_matrices}, \eqref{Minfty}, \eqref{cyclic}, find a piecewise holomorphic $2\times2$ matrix function $\Psi(\lambda)$ with the following properties:
\begin{enumerate}\itemsep=0pt
\item[$1)$] $\|\Psi(\lambda)\lambda^{\delta\sigma_3}-I\|\leq C|\lambda|^{-1}$, $C={\rm const}$, as $\lambda\to\infty$,
\item[$2)$] $\lim\limits_{\lambda\to\infty} \lambda(\Psi(\lambda)\lambda^{\delta\sigma_3}-I)_{12}=0$,
\item[$3)$] $\|\Psi(\lambda)(\lambda-\lambda_j)^{-\alpha_j\sigma_3}\|\leq C_j$, $C_j={\rm const}$, as $\lambda\to\lambda_j$, where $\lambda_1=0$, $\lambda_2=a$, $\lambda_3=1$,
\item[$4)$] $\|\Psi(\lambda)\|\leq C_0$, $C_0={\rm const}$, as $\lambda$ approaches any nodal point of the graph~$\gamma$,
\item[$5)$] across the oriented contour $\gamma$, the discontinuity condition holds,
\begin{gather*}
\Psi_+(\lambda)=\Psi_-(\lambda)G(\lambda),\qquad \lambda\in\gamma,
\end{gather*}
where $\Psi_+(\lambda)$ and $\Psi_-(\lambda)$ are the left and right continuous limits of $\Psi(\lambda)$ as $\lambda$ approaches the contour~$\gamma$, and $G(\lambda)$ is the piecewise constant matrix defined on~$\gamma$, see Fig.~{\rm \ref{fig1}}.
\end{enumerate}
\end{RHP}

Let us show that, conversely, this RH problem leads to the structure (\ref{A_at_simple_pole+_p6}) of the coefficient matrix~$A(\lambda)$. Let $\Psi(\lambda)$ be a unique solution of the RH Problem~\ref{RHP_at_pole_sigma+}. First, $\det\Psi(\lambda)\equiv1$ since this determinant is piecewise holomorphic, continuous across the graph~$\gamma$, bounded at $\lambda_j$, $j=1,2,3$, and at the nodes of the graph~$\gamma$ and approaches the unit as $\lambda\to\infty$. Consider now the function $A(\lambda)=\Psi_{\lambda}\Psi^{-1}$. It is piecewise holomorphic, continuous across $\gamma$, has simple poles at $\lambda=\lambda_j$, $j=1,2,3$, and $\lambda=\infty$, is bounded at the nodes of $\gamma$ and therefore it is rational.

The conditions (1) and (2) of the RH Problem~\ref{RHP_at_pole_sigma+} imply the asymptotics
\begin{gather*}
\Psi(\lambda)=\bigl(I+\psi_1\lambda^{-1}+{\mathcal O}\big(\lambda^{-2}\big)\bigr)\lambda^{-\delta\sigma_3},\qquad \lambda\to\infty,
\end{gather*}
where $(\psi_1)_{12}=0$. Therefore
\begin{gather*}
A(\lambda)=-\frac{\delta}{\lambda}\sigma_3 -\frac{1}{\lambda^2} \bigl(\psi_1+[\psi_1,\sigma_3] \bigr) +{\mathcal O}\big(\lambda^{-3}\big),\qquad \lambda\to\infty,
\end{gather*}
i.e.,
\begin{gather*}
(A(\lambda))_{12}={\mathcal O}\big(\lambda^{-3}\big).
\end{gather*}
All these properties of $A(\lambda)$ imply its structure given in (\ref{A_at_simple_pole+_p6}). Thus, using (\ref{u_+_def}), the first row of $\Psi(\lambda)$ determines a fundamental system of solutions to GHE~(\ref{Heun_gen_can}).

We have proved the following

\begin{prop}Solution of the RH Problem~{\rm \ref{RHP_at_pole_sigma+}}, if it exists, determines a fundamental system of solutions to GHE~\eqref{Heun_gen_can} with the prescribed monodromy properties.
\end{prop}

The accessory parameter $\nu$ in (\ref{Heun_gen_can}) is also determined via the RH Problem~\ref{RHP_at_pole_sigma+}. Indeed, $\nu$ can be extracted from the asymptotics of $\Psi(\lambda)$ at infinity. Namely, the parameter~$d_1$, i.e., the diagonal part of the term~${\mathcal O}\big(\lambda^{-1}\big)$ of the asymptotic expansion of~$\Psi\lambda^{\delta\sigma_3}$, determines the coefficient $b_3$ in the coefficient matrix~(\ref{A_at_simple_pole+_p6}) and hence the free coefficient~$c_0$ in the Laurent expansion~(\ref{simple_pole_p6}) and the accessory parameter~$\nu$,
\begin{gather}
b_3=-d_1+\delta(a+1),\qquad c_0=\frac{b_3+a-1}{2\delta-1}=\frac{-d_1+a(\delta+1)+\delta-1}{2\delta-1},\nonumber\\
\nu=-d_1(2\delta-1)+\alpha_1+\alpha_2-(\alpha_1+\alpha_2)^2+\alpha_3^2-\delta^2\nonumber\\
\hphantom{\nu=}{}+a\big( \alpha_1+\alpha_3-(\alpha_1+\alpha_3)^2+\alpha_2^2-\delta^2\big)+\delta(2\delta-1)(a+1).\label{b3_c0_nu_via_d1_sigma+}
\end{gather}

\begin{rem}The condition (2) of the RH Problem~\ref{RHP_at_pole_sigma+} can be in fact thought of as an addition to cyclic relation~(\ref{cyclic}) restriction on the monodromy matrices $\{M_j\}$ which would also involve the point $a$. This restriction can be formulated in terms of the sixth Painlev\'e transcendent in two different but equivalent ways. Firstly, introducing on the monodromy surface~$\mathcal{M}$ coordinates $t$ and $s$ (see~(\ref{tsdef})), let $y(x) \equiv y(x;t,s) $ be the corresponding sixth Painlev\'e function (cf.~(\ref{yts})). Then the point $a$ must be one of the $(\sigma = +1)$ poles of $y(x;t,s) $. Alternatively, assuming the position $a$ of the pole of $y(x)$ to be a free parameter, one can parameterize $y(x)$ by the pair{\footnote{This parameterization of $y(x)$ is more subtle than by the monodromy pair $(t,s)$. Indeed, \PVI\ function might have infinitely many poles and therefore one might have an infinite discrete set of the pairs $(t, a_n)$ corresponding to the same \PVI\ function $y(x) \equiv y(x;t,s) $. Hence, one has to be careful describing the global properties of the parameterization, $y(x) \equiv y(x;t,a)$. Also, the function $s(t,a)$ implicitly defined by equation~(\ref{sta}) below might have infinitely many branches. This means that for each pair $(t,a)$ there might be an infinite discrete set of the values of the second monodromy coordinate, $s$, and hence an infinite discrete set of the \PVI\ functions $y(x)$ with the same values of $t$ and $a$. The discussion of these important issues is, however, beyond the scope of this work.}} $(t, a)$, $y(x) \equiv y(x;t,a)$. Then, the second monodromy data~$s$ becomes the function of $(t,a)$ which can be described implicitly as follows. Note that together with~$s$, the coefficient~$c_0$ in the Laurent expansion~(\ref{simple_pole_p6}) and the accessory parameter $\nu$ also become the functions of~$t$ and~$a$,
\begin{gather*}
s \equiv s(t,a),\qquad c_0 \equiv c_0(t,a), \qquad \nu\equiv \nu(t,a).
\end{gather*}
At the same time, the RH Problem~\ref{RHP_p6} determines the solution $\Psi(\lambda)$ and hence all the objects related to it, specifically the coefficients of its expansion~(\ref{Psi_p6_as_at_8}) at $\lambda = \infty$, as the functions of the point on the monodromy surface, that is as the functions of~$t$ and~$s$. In particular, we have that $d_1 = d_1(x; t, s)$. Using now the second equation in~(\ref{b3_c0_nu_via_d1_sigma+}), the function~$s(t,a)$ can be defined implicitly via the equation
\begin{gather}\label{sta}
c_0(t,a)=\frac{-d_1(a; t, s)+a(\delta+1)+\delta-1}{2\delta-1}.
\end{gather}
\end{rem}
\begin{rem} Excluding the parameter $d_1$ from the second and third equations in (\ref{b3_c0_nu_via_d1_sigma+}), we obtain the formula relating the accessory parameter~$\nu(t, a)$ and the free coefficient~$c_0(t,a)$ in the Laurent expansion~(\ref{simple_pole_p6}),
\begin{gather}
\nu(t,a) = (1-2\delta)^2c_0(t,a) + 3\delta^2(1+a)\nonumber\\
 \hphantom{\nu(t,a) =}{} + a\big(\alpha_1 + \alpha_3 -(\alpha_1 + \alpha_3)^2 + \alpha_2^2\big) + \alpha_1 + \alpha_2 -(\alpha_1 + \alpha_2)^2 + \alpha_3^2.\label{vc0}
\end{gather}
As it has already been mentioned before, this relation has been already found in \cite{LLNZ} using a~heuristic technique based on the remarkable connection (also discovered in~\cite{LLNZ}) between the classical conformal blocks and the sixth Painlev\'e equation.
\end{rem}

\begin{rem} In this section we considered the regular case, i.e., $\delta \neq 0, \frac{1}{2}$ and $\sigma= +1$ only. Our arguments, however, can be easily extended to the generic singular case, i.e., $\delta \neq 0, \frac{1}{2}, 1$ and $\sigma= -1$, and to the special singular cases, i.e., $\delta = 1$ and $\sigma= -1$ and $\delta = \frac{1}{2}$. One only needs, before producing the relevant analogs of the RH problem \ref{RHP_at_pole_sigma+}, to make the preliminary Schlesinger transformations of the RH Problem~\ref{RHP_p6} with the gauge matrices~$R(\lambda)$ discussed in details in Sections~\ref{2.3.2}, \ref{2.3.3} and~\ref{2.3.4}.
\end{rem}
\begin{rem} In this paper we are dealing with the poles of $y(x)$. Similar results concerning the reduction of the RH Problem~\ref{RHP_p6} to a Riemann--Hilbert problem for Heun equation can be obtained for two other critical values of $y(x)$, i.e., when $a$ is either a zero of $y(x)$ or $y(a) =1$.
\end{rem}

We believe that the Riemann--Hilbert technique we are developing here can be used to study effectively the Heun functions. In particular, the relation (\ref{vc0}) allows one to get nontrivial information about the accessory parameter $\nu(t,a)$. In particular, using the known connection formulae for the pole distributions of the \PVI\ equation~\cite{Guz} (obtained with the help of the isomonodromy RH Problem~\ref{RHP_p6}; see also~\cite{Guz2} for complete list of the asymptotic connection formulae and the history of the question), one can obtain the asymptotic expansion of the Laurent coefficient $c_0(t,a)$ for the either large values of~$a$ or for small values of~$a$ or for the values of~$a$ close to~1. This in turn would yield the explicit formulae for the asymptotic behavior of the accessory parameter $\nu(t, a)$ as $ a\sim \infty$, $a\sim 0$ and $a \sim 1$. This question should be addressed in details in the future work on this subject. In the rest of this paper, we will demonstrate the usefulness of the RH Problem~\ref{RHP_at_pole_sigma+} in the study of another issue related to the Heun equations which is the construction of its explicit solutions.

\subsection[Example: reducible monodromy, generalized Jacobi and Heun polynomials]{Example: reducible monodromy, generalized Jacobi\\ and Heun polynomials}

Consider the case of reducible monodromy when all the monodromy matrices are upper triangular. Reducible monodromy for \PVI\ system was studied, e.g., in~\cite{Mazz} where it was shown that this class of Painlev\'e functions contains classical and all rational solutions to~\PVI.

In \cite{FIKN}, the relevant function $\Psi(\lambda)$ was constructed explicitly for $0\leq\Re\alpha_j<\frac{1}{2}$, $j=1,2,3$, and $\alpha_1+\alpha_2+\alpha_3+\delta=0$. Below, we use the approach of~\cite{FIKN} assuming that
\begin{gather}\label{alpha_delta_assumption}
\Re\alpha_j\in\big[0,\tfrac{1}{2}\big),\qquad \alpha_1+\alpha_2+\alpha_3+\delta=-n,\qquad n\in{\mathbb N}.
\end{gather}

Let all the monodromy matrices $M_j$, $j=1,2,3$, be upper triangular and thus each of them depends on one free parameter~$s_j$,
\begin{gather}\label{reducible_monodromy}
M_j=\begin{pmatrix} {\rm e}^{-2\pi {\rm i}\alpha_j}&s_j\\ 0&{\rm e}^{2\pi {\rm i}\alpha_j} \end{pmatrix}, \qquad s_j\neq0,\qquad j=1,2,3.
\end{gather}
The cyclic relation $M_1M_2M_3=M_{\infty}$ implies that
\begin{gather*}
\alpha_1+\alpha_2+\alpha_3+\delta=-n \in{\mathbb Z},\qquad s_1{\rm e}^{2\pi {\rm i}\alpha_2} +s_2{\rm e}^{-2\pi {\rm i}\alpha_1} +s_3{\rm e}^{2\pi {\rm i}\delta}=0.
\end{gather*}
Thus the set of reducible monodromy data form a 2-dimensional linear space.

Following \cite{FIKN}, we first simplify the RH jump graph replacing $\gamma$ by the broken line $[\lambda_1,\lambda_2]\cup[\lambda_2,\lambda_3] \cup[\lambda_3,\infty)$, see Fig.~\ref{fig1_simple}.

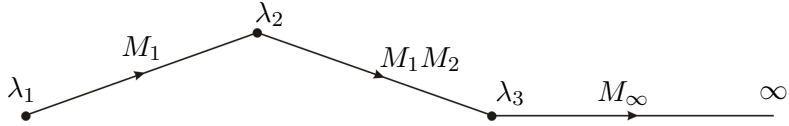
\begin{figure}[t]\centering
\begin{tikzpicture}

\node at (0,0) {\includegraphics[scale=0.6]{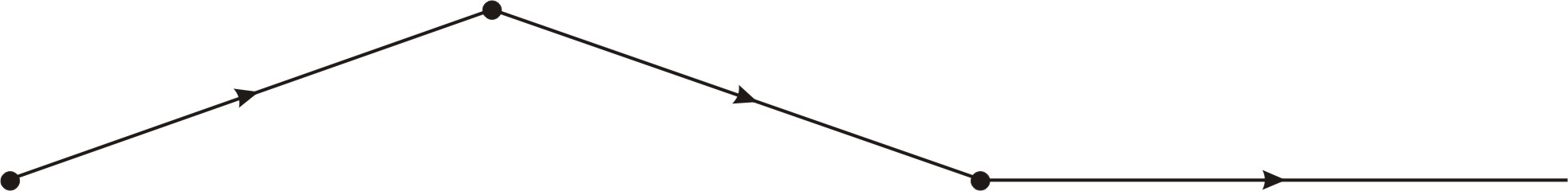}};

\node at (-5,-0.2) {$\lambda_1$}; \node at (-3.4,0.35) {$M_1$};

\node at (-1.7,0.8) {$\lambda_2$}; \node at (0.3,0.2) {$M_1M_2$};

\node at (1.5,-0.25) {$\lambda_3$}; \node at (3,-0.25) {$M_\infty$}; \node at (5,-0.25) {$\infty$};

\end{tikzpicture}
\caption{Simplified jump contour for RH Problem~\ref{RHP_p6}.} \label{fig1_simple}
\end{figure}

On the plane, make a cut along the broken line $[\lambda_1,\lambda_2]\cup[\lambda_2,\lambda_3]\cup[\lambda_3,\infty)$, define the function
\begin{gather*}
f(\lambda)= (\lambda-\lambda_1)^{\alpha_1} (\lambda-\lambda_2)^{\alpha_2} (\lambda-\lambda_3)^{\alpha_3}.
\end{gather*}
Although until we reach Proposition \ref{3.7} it is not really important, we remind that
\begin{gather*}
\lambda_1 = 0, \quad \lambda_2 = x, \qquad\mbox{and}\qquad \lambda_3 =1.
\end{gather*}
Observe the following properties of $f(\lambda)$:
\begin{gather*}
f(\lambda)=\lambda^{-\delta-n} \big(1+{\mathcal O}\big(\lambda^{-1}\big)\big),\qquad \lambda\to\infty,\nonumber\\
f_+(\lambda)=f_-(\lambda){\rm e}^{-2\pi {\rm i}\alpha_1},\qquad \lambda\in(\lambda_1,\lambda_2),\nonumber\\
f_+(\lambda)=f_-(\lambda){\rm e}^{-2\pi {\rm i}(\alpha_1+\alpha_2)},\qquad \lambda\in(\lambda_2,\lambda_3),\nonumber\\
f_+(\lambda)=f_-(\lambda){\rm e}^{2\pi {\rm i}\delta},\qquad \lambda\in(\lambda_3,\infty).
\end{gather*}

Let us represent the solution $\Psi(\lambda)$ as the product
\begin{gather*}
\Psi(\lambda)=\Phi(\lambda)f^{\sigma_3}(\lambda).
\end{gather*}
The function $\Phi(\lambda)$ thus has the following properties:

\begin{RHP}\label{RHP_Jacobi}\quad
\begin{enumerate}\itemsep=0pt
\item[$1)$] $\Phi(\lambda)=\big(I+{\mathcal O}\big(\lambda^{-1}\big)\big) \lambda^{n\sigma_3}$ as $\lambda\to\infty$,
\item[$2)$] $\|\Phi(\lambda)f^{\sigma_3}(\lambda)E_j f^{-\sigma_3}(\lambda)\|\leq C$ as $\lambda\to\lambda_j$, $j=1,2,3$,
\item[$3)$] $\Phi_+(\lambda)=\Phi_-(\lambda)G_f(\lambda)$, $\lambda\in(\lambda_1,\lambda_2)\cup (\lambda_2,\lambda_3)$, moreover
\begin{gather}
G_f(\lambda)=I+g(\lambda)\sigma_+,\qquad g(\lambda)= \begin{cases}
s_1f_+(\lambda)f_-(\lambda),& \lambda\in(\lambda_1,\lambda_2),\\
-s_3{\rm e}^{2\pi {\rm i}\delta} f_+(\lambda)f_-(\lambda),& \lambda\in(\lambda_2,\lambda_3).
\end{cases}\label{Gf_reducible_def}
\end{gather}
\end{enumerate}
\end{RHP}

\begin{prop}Solution to the RH Problem~{\rm \ref{RHP_Jacobi}} is given by polynomials orthogonal with respect to the weight function $g(\lambda)$ on $(\lambda_1,\lambda_2)\cup(\lambda_2,\lambda_3)$ and by their Cauchy integrals.
\end{prop}

\begin{proof}First of all, by the conventional arguments, if a solution to this problem exists, it is unique.

Next, upper triangularity of all jump matrices, see condition (3), means that the first column of $\Phi(\lambda)$ is single-valued and continuous across the broken line $(\lambda_1,\lambda_2)\cup(\lambda_2,\lambda_3)$. Condition (2) means that the first column is bounded at $\lambda=\lambda_j$, $j=1,2,3$. Then condition~(1) yields that the entries of the first column are some polynomials of degree $n$ and $n-1$, respectively.

Denote the relevant monic polynomials of degree $n$ and $n-1$ by $\pi_n(\lambda)$ and $\pi_{n-1}(\lambda)$, respectively. Consider an auxiliary matrix function $Y(\lambda)$ (cf.~\cite{FIK})
\begin{gather*}
Y(\lambda)=\begin{pmatrix}
\pi_n(\lambda)& \displaystyle \frac{1}{2\pi {\rm i}}\int_{\ell}\pi_n(\zeta)g(\zeta)\frac{\dd\zeta}{\zeta-\lambda}\vspace{1mm}\\
c_{n-1}\pi_{n-1}(\lambda)& \displaystyle \frac{c_{n-1}}{2\pi {\rm i}} \int_{\ell}\pi_{n-1}(\zeta)g(\zeta)\frac{\dd\zeta}{\zeta-\lambda}
\end{pmatrix},
\end{gather*}
where $\ell=(\lambda_1,\lambda_2)\cup(\lambda_2,\lambda_3)$ and the constant $c_{n-1}$ is defined by
\begin{gather*}
c_{n-1}=-\frac{2\pi {\rm i}} {\int_{\ell}\pi_{n-1}^2(\zeta)g(\zeta)\dd \zeta}.
\end{gather*}

As it is easy to see, $Y(\lambda)$ satisfies the jump condition~(3).

At the points $\lambda=\lambda_j$, $j=1,2,3$, the function $Y(\lambda)$ has the algebraic branching, $(\lambda-\lambda_j)^{2\alpha_j}$, and therefore, taking into account assumption (\ref{alpha_delta_assumption}), condition (2) is satisfied.

Finally, consider the series expansion of the entries of the second column of $Y(\lambda)$ as $|\lambda|$ is large enough,
\begin{gather*}
\frac{1}{2\pi {\rm i}}\int_{\ell}\pi_m(\zeta)g(\zeta) \frac{\dd\zeta}{\zeta-\lambda}=-\frac{1}{2\pi {\rm i}}\sum_{k=0}^{\infty}\lambda^{-k-1}\int_{\ell} \pi_m(\zeta)\zeta^kg(\zeta)\dd \zeta,
\end{gather*}
$m=n$ or $m=n-1$. Then condition (1) means that the polynomials $p_n(\lambda)$ and $p_{n-1}(\lambda)$ are both orthogonal to all lower degree monomials $\lambda^k$ for $k=0,1,\dots,n-1$ and $k=0,1,\dots,n-2$, respectively,
\begin{gather*}
\int_{\ell} \pi_n(\zeta)\zeta^k g(\zeta) \dd \zeta=0,\qquad k=0,1,\dots,n-1,\\
\int_{\ell}\pi_{n-1}(\zeta)\zeta^k g(\zeta) \dd\zeta=0,\qquad k=0,1,\dots,n-2.
\end{gather*}
The choice of $c_{n-1}$ made above implies the normalization of $Y(\lambda)$ at infinity
\begin{gather*}
Y(\lambda)=\big(I+{\mathcal O}\big(\lambda^{-1}\big)\big) \lambda^{n\sigma_3}
\end{gather*}
that complies with the condition (1). Thus the explicitly constructed function $Y(\lambda)$ solves the RH Problem~\ref{RHP_Jacobi}. Since its solution is unique, proof of proposition is completed.
\end{proof}

The weight function $g(\lambda)$ in (\ref{Gf_reducible_def}) can be understood as a generalization of the hypergeometric weight, thus we call the polynomials $p_n(\lambda)$ and $p_{n-1}(\lambda)$ determined by the RH Problem~\ref{RHP_Jacobi} the {\em generalized Jacobi} polynomials.

Now, we are going to construct the generalized Jacobi polynomials explicitly and relate them to the polynomial solutions of the Heun equation, i.e., to the {\it Heun polynomials.}

To this end, we look for solution to the RH Problem~\ref{RHP_Jacobi} in the form
\begin{gather}\label{Phi_reducible_ansatz}
\Phi(\lambda)=R(\lambda)
\begin{pmatrix}
1&\phi(\lambda)\\
0&1
\end{pmatrix},
\end{gather}
where $R(\lambda)$ is a matrix-valued polynomial. Using the jump properties of $\Phi(\lambda)$, we find the jump properties of the scalar function~$\phi(\lambda)$,
\begin{gather*}
\phi_+(\lambda)-\phi_-(\lambda)=g(\lambda),\qquad \lambda\in (\lambda_1,\lambda_2)\cup(\lambda_2,\lambda_3).
\end{gather*}
One of the solutions to this scalar jump problem is given explicitly
\begin{gather}\label{phi_reducible_sol}
\phi(\lambda)=\frac{1}{2\pi {\rm i}}\int_{\lambda_1}^{\lambda_3}\frac{g(\zeta)}{\zeta-\lambda}\dd\zeta.
\end{gather}
Observe the behavior of $\phi(\lambda)$~(\ref{phi_reducible_sol}) at the singularities
\begin{gather*}
\phi(\lambda)= {\mathcal O}\big(\lambda^{-1}\big),\qquad \lambda\to\infty,\nonumber\\
\phi(\lambda)={\mathcal O}\big((\lambda-\lambda_j)^{2\alpha_j}\big) +{\mathcal O}(1),\qquad \lambda\to\lambda_j,\qquad j=1,2,3. 
\end{gather*}
Below, we use the coefficients $\phi_k$ of the expansion of $\phi(\lambda)$~(\ref{phi_reducible_sol}) near infinity
\begin{gather*}
\phi(\lambda)= \frac{1}{2\pi {\rm i}}\int_{\lambda_1}^{\lambda_3}\frac{g(\zeta)}{\zeta-\lambda}\dd\zeta= \sum_{k=1}^{\infty}\phi_k\lambda^{-k},
\end{gather*}
closely related to the moments of the weight function $g(\lambda)$,
\begin{gather}\label{phi_f_as}
\phi_k= -\frac{1}{2\pi {\rm i}}\int_{\lambda_1}^{\lambda_3}g(\zeta)\zeta^{k-1}\dd\zeta,\qquad k\in{\mathbb N}.
\end{gather}

The left factor $R(\lambda)$ is a polynomial matrix of the Schlesinger transformation at infinity~\cite{JM} and for $n\geq0$ it can be found explicitly. For instance,
\begin{gather}
\text{if}\ n=0\colon\quad R(\lambda)=R_0(\lambda)=I,\qquad \pi_0(\lambda)\equiv1,\qquad \pi_{-1}(\lambda)\equiv0,\nonumber\\
\text{if}\quad n=1\colon\quad R(\lambda)=R_1(\lambda)=\begin{pmatrix}
\lambda-\dfrac{\phi_2}{\phi_1}&-\phi_1 \vspace{1mm}\\
\dfrac{1}{\phi_1}&0 \end{pmatrix},\qquad \pi_1(\lambda)=\lambda-\frac{\phi_2}{\phi_1}. \label{R_n=0_n=1}
\end{gather}
In particular, relations (\ref{R_n=0_n=1}) imply that the RH Problem~\ref{RHP_Jacobi} for $n=0$ is always solvable, cf.~\cite{FIKN}, while for $n=1$, this problem is solvable if $\phi_1\neq0$.

To formulate the result for any fixed $n\geq2$, introduce the following explicit form of the polynomial matrix $R(\lambda)=R_n(\lambda)$,
\begin{gather}\label{R_explicit}
R_n(\lambda)=
\begin{pmatrix}
\displaystyle \sum_{k=0}^np_k^{(n)}\lambda^k& \displaystyle \sum_{k=0}^{n-1}q_k^{(n)}\lambda^k\vspace{1mm}\\
\displaystyle \sum_{k=0}^{n-1}r_k^{(n)}\lambda^k& \displaystyle \sum_{k=0}^{n-2}s_k^{(n)}\lambda^k
\end{pmatrix} ,\qquad p_n^{(n)}=1,
\end{gather}
and the column vectors of the coefficients of the polynomial entries of $R_n(\lambda)$,
\begin{gather*}
{\bf p}_n=
\begin{pmatrix}
p_0^{(n)}\\
\vdots\\
p_{n-1}^{(n)}
\end{pmatrix},\qquad
{\bf q}_n=
\begin{pmatrix}
q_0^{(n)}\\
\vdots\\
q_{n-1}^{(n)}
\end{pmatrix}
,\qquad
{\bf r}_n=
\begin{pmatrix}
r_0^{(n)}\\
\vdots\\
r_{n-1}^{(n)}
\end{pmatrix}
,\qquad
{\bf s}_n=
\begin{pmatrix}
s_0^{(n)}\\
\vdots\\
s_{n-2}^{(n)}
\end{pmatrix}.
\end{gather*}

We also need the Hankel matrix ${\mathcal H}_n=\{\phi_{i+j-1}\}_{i,j=1,n}$ of the moments~\eqref{phi_f_as} along with its determinant $\triangle_n$,
\begin{gather*}
{\mathcal H}_n= \begin{pmatrix}
\phi_1&\phi_2&\cdots&\phi_n\\
\phi_2&\phi_3&\cdots&\phi_{n+1}\\
\vdots&\vdots&\ddots&\vdots\\
\phi_n&\phi_{n+1}&\cdots&\phi_{2n-1}\\
\end{pmatrix},\qquad
\triangle_n=\det{\mathcal H}_n.
\end{gather*}
Finally define the sequence of coefficients $f_k$ for asymptotic expansion of $f(\lambda)=\prod\limits_{j=1,2,3}(\lambda-\lambda_j)^{\alpha_j}$ at infinity
\begin{gather*}
f(\lambda)= \lambda^{-\delta-n}{\rm e}^{\sum\limits_{k=1}^{\infty}f_k\lambda^{-k}},\qquad f_k=-\frac{1}{k} \sum_{j=1}^3\alpha_j\lambda_j^k,\qquad \lambda\to\infty.
\end{gather*}

Then we have the following
\begin{prop}\looseness=-1 RH Problem~{\rm \ref{RHP_Jacobi}} is solvable if and only if $\triangle_n\neq0$. Its solution has the form~\eqref{Phi_reducible_ansatz} where the coefficients of the polynomial entries of the matrix~$R(\lambda)$~\eqref{R_explicit} are given by
\begin{gather}
{\bf p}_n=-{\mathcal H}_n^{-1}\begin{pmatrix} \phi_{n+1}\\ \vdots\\ \phi_{2n} \end{pmatrix},\qquad p_n^{(n)}=1,\nonumber\\ q_k^{(n)}= -\sum_{m=k+1}^n \phi_{m-k}p_m^{(n)},\qquad k=0,\dots,n-1,\label{p_n_eq}\\
{\bf r}_n={\mathcal H}_n^{-1} \begin{pmatrix} 0\\ \vdots\\ 0\\ 1 \end{pmatrix},\qquad r_{n-1}^{(n)}= c_{n-1},\qquad s_k^{(n)}=-\sum_{m=k+1}^{n-1} \phi_{m-k}r_m^{(n)},\qquad k=0,\dots,n-2.\nonumber
\end{gather}
\end{prop}
\begin{proof} Requiring that the product (\ref{Phi_reducible_ansatz}) has the canonical asymptotics
\begin{gather*}
R_n(\lambda) \begin{pmatrix} 1&\phi(\lambda) \\ 0&1 \end{pmatrix} =
\begin{pmatrix} \lambda^n+{\mathcal O}\big(\lambda^{n-1}\big)& {\mathcal O}\big(\lambda^{-n-1}\big)\\
{\mathcal O}\big(\lambda^{n-1}\big)& \lambda^{-n}+{\mathcal O}\big(\lambda^{-n-1}\big)
\end{pmatrix},
\end{gather*}
we find the following expansions in the second column of the product
\begin{gather}
(12)\colon\quad \sum_{k=0}^{n-1}q_k^{(n)}\lambda^k +\sum_{m=0}^n\sum_{l=1}^{\infty} \phi_lp_m^{(n)}\lambda^{m-l}\nonumber\\
\qquad\quad \ {} =\sum_{k=0}^{n-1}\lambda^k\left(q_k^{(n)}+\sum_{m=k+1}^n\phi_{m-k}p_m^{(n)}\right)+\sum_{k=1}^{\infty}\lambda^{-k}\sum_{m=0}^n\phi_{k+m}p_m^{(n)}, \nonumber\\
(22)\colon\quad \sum_{k=0}^{n-2}s_k^{(n)}\lambda^k +\sum_{m=0}^{n-1}\sum_{l=1}^{\infty} \phi_lr_m^{(n)}\lambda^{m-l}\nonumber\\
\qquad\quad \ {} =\sum_{k=0}^{n-2}\lambda^k\left(s_k^{(n)}+\sum_{m=k+1}^{n-1}\phi_{m-k}r_m^{(n)}\right)+\sum_{k=1}^{\infty}\lambda^{-k}\sum_{m=0}^{n-1}\phi_{k+m}r_m^{(n)}.\label{coeffs_eqs_ori}
\end{gather}
Thus expressions for $q_k^{(n)}$ and $s_k^{(n)}$ in (\ref{coeffs_eqs}) in terms of $p_m^{(n)}$ and $r_m^{(n)}$ come from (\ref{coeffs_eqs_ori}) at the non-negative degrees of $\lambda$.

Evaluating the prescribed terms at all orders from $\lambda^{-1}$ to $\lambda^{-n}$, we find equations for the vector coefficients ${\bf r}_n$ and ${\bf p}_n$. Combining them into the matrix form, it follows
\begin{gather}\label{coeffs_eqs}
{\mathcal H}_n{\bf p}_n
+(\phi_{n+1},\phi_{n+2},\dots,\phi_{2n})^{\rm T}=0, \qquad {\mathcal H}_n{\bf r}_n=(0,\dots,0,1)^{\rm T}.
\end{gather}
If $\triangle_n\neq0$, the moment matrix ${\mathcal H}_n$ is invertible, and the coefficient vectors ${\bf p}_n$, ${\bf r}_n$ are computed by~(\ref{p_n_eq}).

If $\triangle_n=0$ then for the equations (\ref{coeffs_eqs}) on the vectors ${\bf p}_n$ and ${\bf r}_n$, there are two alternatives. The first one implies that $R(\lambda)$ does not exist. The second alternative implies an infinite number of $R(\lambda)$ and therefore contradicts the uniqueness of solution to the RH Problem~\ref{RHP_Jacobi}.

This completes the proof.
\end{proof}

In the next proposition, we find the classical sixth Painlev\'e functions determined by the explicitly solvable RH Problem~\ref{RHP_Jacobi}. For instance, in the simplest cases,
\begin{gather*}
\text{if}\ n=0\colon\quad y=x+1 -\frac{(\delta-1)\phi_2}{\big(\delta-\frac{1}{2}\big)\phi_1}-\frac{f_1}{\delta-\frac{1}{2}},\nonumber\\
\text{if}\ n=1\colon\quad y=x+1 -\frac{\delta-1}{\delta-\frac{1}{2}}\frac{\left|\begin{matrix}\phi_1&\phi_2\\ \phi_3&\phi_4\end{matrix}\right|}{\left|
\begin{matrix}\phi_1&\phi_2\\ \phi_2&\phi_3 \end{matrix}\right|}+\frac{\delta}{\delta-\frac{1}{2}}\frac{\phi_2}{\phi_1}-\frac{f_1}{\delta-\frac{1}{2}}.
\end{gather*}

\begin{prop}\label{3.7}If $\triangle_n\triangle_{n+1}\neq0$ then the RH Problem~{\rm \ref{RHP_Jacobi}} determines the classical solution to \PVI\ corresponding to the~$\Psi$ function with the prescribed reducible monodromy data in terms of the moment functions~$\phi_k$, $k=1,\dots,\phi_{2n+2}$ as follows
\begin{gather}
y=x+1 -\frac{f_1}{\delta-\frac{1}{2}}-\frac{\delta-1}{\delta-\frac{1}{2}}\frac{\triangle_n}{\triangle_{n+1}}\phi_{2n+2}
+\frac{\delta-1}{\delta-\frac{1}{2}}\frac{\triangle_n}{\triangle_{n+1}}\begin{pmatrix}\phi_{n+2}&\dots&\phi_{2n+1}\end{pmatrix}{\mathcal H}_n^{-1}\begin{pmatrix}\phi_{n+1}\\ \vdots\\ \phi_{2n}\end{pmatrix}\nonumber\\
\hphantom{y=}{} +\frac{\delta}{\delta-\frac{1}{2}}\begin{pmatrix}
\phi_{n+1}&
\dots&
\phi_{2n}
\end{pmatrix}
{\mathcal H}_n^{-1}
\begin{pmatrix}
0\\
\vdots\\
0\\
1
\end{pmatrix}.\label{y_from_phi}
\end{gather}
\end{prop}

\begin{proof}According to (\ref{p6_via_RH_sol}), in order to find the Painlev\'e function $y$ we have to compute two first coefficients $(\psi_1)_+$ and $(\psi_2)_+$ along with the coefficient $d_1$ of the asymptotic expansion of~$\Psi(\lambda)$ at $\lambda\to\infty$, see (\ref{Psi_p6_as_at_8}). Using~(\ref{coeffs_eqs_ori}) and~(\ref{coeffs_eqs}) again, we find
\begin{gather*}
\Psi(\lambda)= R_n(\lambda) \begin{pmatrix} 1&\phi(\lambda)\\ 0&1 \end{pmatrix} f^{\sigma_3}(\lambda)\nonumber\\
\hphantom{\Psi(\lambda)}{}= \begin{pmatrix}
\displaystyle \sum_{k=0}^n \lambda^{-k} p_{n-k}^{(n)} &
\displaystyle \sum_{k=1}^{\infty} \lambda^{-k} \sum_{m=0}^n \phi_{k+n+m}p_m^{(n)}\vspace{1mm}\\
\displaystyle \sum_{k=1}^{n} \lambda^{-k} r_{n-k}^{(n)} &
\displaystyle 1 +\sum_{k=1}^{\infty} \lambda^{-k} \sum_{m=0}^{n-1} \phi_{k+n+m}r_m^{(n)}
\end{pmatrix}
{\rm e}^{\sum\limits_{k=1}^{\infty}\lambda^{-k}f_k\sigma_3} \lambda^{-\delta\sigma_3} \nonumber\\
\hphantom{\Psi(\lambda)}{} = \begin{pmatrix}
1 & \dfrac{1}{\lambda} (\psi_1)_+ +\dfrac{1}{\lambda^2} (\psi_2)_+ +{\mathcal O}\left(\dfrac{1}{\lambda^3}\right)\vspace{1mm}\\
\dfrac{1}{\lambda} r_{n-1}^{(n)} +{\mathcal O}\left(\dfrac{1}{\lambda^2}\right) & 1 \end{pmatrix} \nonumber\\
\hphantom{\Psi(\lambda)=}{}\times\left(I+{\mathcal O}\left(\frac{1}{\lambda^2}\right)\right){\rm e}^{\frac{1}{\lambda}(f_1+p_{n-1}^{(n)})\sigma_3}
\lambda^{-\delta\sigma_3} ,
\end{gather*}
where
\begin{gather}
(\psi_1)_+= \sum_{m=0}^n \phi_{n+m+1}p_m^{(n)}= -\begin{pmatrix} \phi_{n+1}& \phi_{n+2}& \cdots& \phi_{2n} \end{pmatrix}
{\mathcal H}_n^{-1} \begin{pmatrix} \phi_{n+1}\\ \phi_{n+2}\\ \vdots\\ \phi_{2n} \end{pmatrix} +\phi_{2n+1}= \frac{\triangle_{n+1}}{\triangle_n},\nonumber\\
(\psi_2)_+= \sum_{m=0}^n \phi_{n+m+2}p_m^{(n)} -\sum_{m=0}^{n-1} \phi_{n+m+1}r_m^{(n)} \sum_{k=0}^n \phi_{n+k+1}p_k^{(n)}\nonumber\\
\hphantom{(\psi_2)_+}{} = \phi_{2n+2}
+\begin{pmatrix} \phi_{n+2}& \phi_{n+3}&\dots& \phi_{2n+1} \end{pmatrix} {\bf p}_n -\phi_{2n+1}
\begin{pmatrix} \phi_{n+1}& \phi_{n+2}&\dots& \phi_{2n} \end{pmatrix} {\bf r}_n \nonumber\\
\hphantom{(\psi_2)_+=}{} -\begin{pmatrix} \phi_{n+1}& \phi_{n+2}& \dots& \phi_{2n} \end{pmatrix} {\bf r}_n \cdot
\begin{pmatrix} \phi_{n+1}& \phi_{n+2}& \dots& \phi_{2n} \end{pmatrix} {\bf p}_n \nonumber\\
\hphantom{(\psi_2)_+}{} = \phi_{2n+2} -\begin{pmatrix} \phi_{n+2}& \phi_{n+3}&\dots& \phi_{2n+1} \end{pmatrix} {\mathcal H}_n^{-1}
\begin{pmatrix} \phi_{n+1}\\ \vdots\\ \phi_{2n} \end{pmatrix}\nonumber\\
\hphantom{(\psi_2)_+=}{} -\frac{\triangle_{n+1}}{\triangle_n} \begin{pmatrix} \phi_{n+1}& \phi_{n+2}& \dots& \phi_{2n} \end{pmatrix}{\mathcal H}_n^{-1}
\begin{pmatrix} 0\\ \vdots\\ 0\\ 1 \end{pmatrix},\nonumber\\
p_{n-1}^{(n)}= - \begin{pmatrix} \phi_{n+1}& \dots& \phi_{2n} \end{pmatrix} {\mathcal H}_n^{-1}
\begin{pmatrix} 0\\ \vdots\\ 0\\ 1 \end{pmatrix}.\label{part_dets}
\end{gather}
Finally, elementary manipulations yield (\ref{y_from_phi}).
\end{proof}

\begin{prop} If $\triangle_{n}\neq0$ and $\triangle_{n+1}=0$, then the RH Problem~{\rm \ref{RHP_Jacobi}} determines a polynomial solution of the Heun equation $($the Heun polynomial$)$ of degree~$n$.
\end{prop}
\begin{proof}The coefficient $(\psi_1)_+$ is computed in (\ref{part_dets}),
\begin{gather}\label{zeros_for_Heun_n}
(\psi_1)_+=\frac{\triangle_{n+1}}{\triangle_{n}}.
\end{gather}
The RH Problem~\ref{RHP_Jacobi} transforms to the Heun RH Problem~\ref{RHP_at_pole_sigma+} if $(\psi_1)_+ = 0$~-- condition (2) of the RH Problem~\ref{RHP_at_pole_sigma+}. In virtue of~(\ref{zeros_for_Heun_n}) this happens only if $\triangle_{n+1} =0$.
\end{proof}

The equation
\begin{gather}\label{hpol}
\triangle_{n+1} =0
\end{gather}
is an equation on the variable $\lambda_2 = x$. This equation determines those $a=x$ for which the Heun equation with the given $\alpha_j$, $\delta$ satisfying~(\ref{alpha_delta_assumption}) admits polynomial solutions given by~$\Phi_{11}(\lambda)$. Alternatively, one can keep $x=a$ free, then (\ref{hpol}) would be equation on the parameters~$s_1$ and~$s_3$ of the monodromy matrices~(\ref{reducible_monodromy}). This in turn would yield a restriction on the accessory parameter of the Heun equation which would guarantee the existence of its polynomial solutions for given~$a$.

\begin{rem} Determinantal conditions on the accessory parameter yielding existence of the polynomial solution of the Heun equation has been known. It would be of course interesting to compare those with the condition which comes from~(\ref{hpol}).
\end{rem}

\begin{rem} There is a vast literature devoted to the exact solutions of the Heun equation~-- see work \cite{M} and references therein. We also want to mention the works \cite{Sm} and \cite{Tak}. In \cite{Sm}, the indicated at the beginning of this paper relation of the Heun equation and the Lame equation has been used to obtain the solvable Heun equations and their polynomial solutions from Lame polynomials corresponding to the finite-gap Lame potentials. In~\cite{Tak}, the hypergeometric type integral representations for the solutions of the Heun equation are found in the case when one of the singularities $\{0, 1, t, \infty\}$ is apparent. It would be important to understand these explicit solutions within the Riemann--Hilbert formalism which we are developing in this article.
\end{rem}

\appendix
\section[Parameterizations of the coefficient matrix for P$_{\text{\rm\scriptsize VI}}$ and GHE]{Parameterizations of the coefficient matrix for P$\boldsymbol{_{\text{\rm\bf \scriptsize VI}}}$ and GHE}\label{appendixA}

\subsection[Parametrization of $A(\lambda)$ for P$_{\text{\rm\scriptsize VI}}$]{Parametrization of $\boldsymbol{A(\lambda)}$ for P$\boldsymbol{_{\text{\rm\bf \scriptsize VI}}}$}\label{A}

Introducing ${\mathfrak s}{\mathfrak l}_2({\mathbb C})$ generators $\sigma_3=\left(\begin{smallmatrix} 1&0\\ 0&-1 \end{smallmatrix}\right)$, $\sigma_+=\left( \begin{smallmatrix}0&1\\ 0&0\end{smallmatrix}\right)$, $\sigma_-=\left(\begin{smallmatrix} 0&0\\ 1&0\end{smallmatrix}\right)$, and taking into account the simplifying assumption that the rational coefficient matrix $A(\lambda)$ is diagonal at infinity, consider the parameterization similar to the one used in~\cite{JM}{\samepage
\begin{gather}\label{parameterization_p6}
A(\lambda)=\frac{-\delta(\lambda-y)^2+p(\lambda-y)+z}{\lambda(\lambda-1)(\lambda-x)}\sigma_3+\frac{\kappa(\lambda-y)}{\lambda(\lambda-1)(\lambda-x)}\sigma_++\frac{\tilde\kappa(\lambda-\tilde y)} {\lambda(\lambda-1)(\lambda-x)}\sigma_-.
\end{gather}
The residue matrices $A_j$ in (\ref{Lax_pair_p6}) have the eigenvalues $\pm\alpha_j$, $j=1,2,3$ respectively.}

As soon as $\delta\neq0$, the parameters $p$, $\tilde y$ and $\tilde\kappa$ can be expressed in terms of the parameters $x$, $y$, $z$, $\kappa$, $\delta$ and the eigenvalues $\alpha_j$, $j=1,2,3$,
\begin{gather*}
p=\frac{\alpha_1^2x}{2\delta y}-\frac{\alpha_3^2(x-1)}{2\delta(y-1)}+\frac{\alpha_2^2x(x-1)}{2\delta(y-x)}-\frac{1}{2}\delta(3y-x-1)-\frac{z^2}{2\delta y(y-1)(y-x)},\\
\tilde y= (y-1)\big(\alpha_1^2x^2-\big(\delta y^2+p y-z\big)^2\big)
\bigl\{ \alpha_1^2 x^2(y - 1) -\alpha_3^2(x - 1)^2 y +(\delta-p)^2 y -4\delta^2 y^2\\
\hphantom{\tilde y=}{} +6\delta p y^2 -p^2 y^2 +6\delta^2 y^3 -4\delta p y^3 -3\delta^2 y^4 -2\delta y z +2\delta y^2 z +z^2 \bigr\}^{-1}, \\
\tilde\kappa= \frac{\alpha_1^2x^2-\big(\delta y^2+p y-z\big)^2} {\kappa y\tilde y}.
\end{gather*}
Considering $y$, $z$ and $\kappa$ as the differentiable functions of $x$, they satisfy the system
\begin{gather}
\frac{\dd(\ln\kappa)}{\dd x}= (2\delta-1)\frac{y-x}{x(x-1)},\nonumber\\
\frac{\dd y}{\dd x}=\frac{y^2-y+2z}{x(x-1)},\nonumber\\
\frac{\dd z}{\dd x} =\frac{1}{x(x-1)y(y-1)(y-x)} \bigl\{ z^2\big(x-2y-2xy+3y^2\big)+zy(y-1)(y-x)(y+x-1) \nonumber\\
\hphantom{\frac{\dd z}{\dd x} =}{} -\alpha_1^2x(y-1)^2(y-x)^2 +\alpha_3^2(x-1)y^2(y-x)^2 -\alpha_2^2x(x-1)y^2(y-1)^2\nonumber\\
\hphantom{\frac{\dd z}{\dd x} =}{} +\delta(\delta-1)y^2(y-1)^2(y-x)^2 \bigr\}.\label{diff_system_p6}
\end{gather}

Finally, eliminating $z$, one arrives at the second order ODE for $y$
\begin{gather*}
y_{xx}=\frac{1}{2} \left( \frac{1}{y} +\frac{1}{y-1} +\frac{1}{y-x}\right) y_x^2 - \left( \frac{1}{x} +\frac{1}{x-1} +\frac{1}{y-x}\right)y_x\\
\hphantom{y_{xx}=}{}
+\frac{y(y-1)(y-x)}{x^2(x-1)^2} \left\{2\left(\delta-\frac{1}{2}\right)^2 -2\alpha_1^2\frac{x}{y^2} +2\alpha_3^2\frac{x-1}{(y-1)^2}
-2\left(\alpha_2^2-\frac{1}{4}\right)\frac{x(x-1)}{(y-x)^2} \right\}, \nonumber
\end{gather*}
which is the classical equation \PVI.

\subsection[Parameterization of $A(\lambda)$ at the pole of $y(x)$ as $\delta\neq0,\frac{1}{2}$, and $\sigma=+1$]{Parameterization of $\boldsymbol{A(\lambda)}$ at the pole of $\boldsymbol{y(x)}$ as $\boldsymbol{\delta\neq0,\frac{1}{2}}$, and $\boldsymbol{\sigma=+1}$}\label{A1}

The coefficients of the matrix $A(\lambda)$ (\ref{A_at_simple_pole+_p6}) are expressed explicitly in terms of the pole position $a$, free parameter~$c_0$ of the Laurent expansion at the pole and the local monodromies $\alpha_j$, $j=1,2,3$, as follows
\begin{gather}
a_3=-\delta,\qquad b_3=c_0(2\delta-1)+1-a,\nonumber\\
c_3=\frac{1}{2\delta} \bigl[ (\delta - 1-c_0(2\delta - 1))^2 -a\alpha_1^2 +(a-1)\alpha_3^2 -a(a-1)\alpha_2^2\nonumber\\
\hphantom{c_3=}{} +a\big(\delta^2-2-2c_0\big(2\delta^2+\delta-1\big)\big) + a^2(\delta+1)^2\bigr],\nonumber\\
c_+=-\kappa_0 \frac{a(a-1)}{2\delta-1}, \nonumber\\
b_-= \frac{2\delta-1}{\kappa_0\delta a(a-1)} \bigl\{ a^3\big(\alpha_2^2-(1+\delta)^2\big) +a^2 \bigl[
3-2\alpha_2^2-3c_0+\alpha_2^2c_0+2\delta+\alpha_2^2\delta \nonumber\\
\hphantom{b_-=}{}+2c_0\delta-2\alpha_2^2c_0\delta -2\delta^2 +7c_0\delta^2-\delta^3+2c_0\delta^3 +\alpha_1^2(1+\delta)
-\alpha_3^2(1+\delta) \bigr]\nonumber\\
\hphantom{b_-=}{} -a \bigl[ \alpha_1^2(1-\delta+c_0(2\delta-1) -\alpha_3^2(2+\delta+c_0(2\delta-1))-\alpha_2^2(1-\delta+c_0(2\delta-1))
\nonumber\\
\hphantom{b_-=}{} +(1-\delta+c_0(2\delta-1)) \big[ 3+\delta-\delta^2+c_0\big({-}3+4\delta+4\delta^2\big) \big] \bigr]\nonumber\\
\hphantom{b_-=}{} +(1+c_0(2\delta-1)) \big[{-}\alpha_3^2+(-1+c_0+\delta-2c_0\delta)^2 \big] \bigr\},\nonumber\\ 
c_-=\frac{2\delta-1}{4\kappa_0\delta^2a(a-1)} \bigl\{ a^4\bigl(\alpha_2^2 - (\delta + 1)^2\bigr)^2
-2a^3\bigl(\alpha_2^2 - (\delta + 1)^2\bigr) \bigl({-}2-\alpha_1^2+\alpha_3^2+\alpha_2^2 +\delta^2\nonumber\\
\hphantom{c_-=}{} +2c_0\big(1-\delta-2\delta^2\big)\bigr) +2a\bigl(\alpha_3^2-(\delta-1+c_0(1-2\delta))^2\bigr)
\bigl(2+\alpha_1^2-\alpha_3^2-\alpha_2^2-\delta^2 \nonumber\\
\hphantom{c_-=}{} +2c_0\big({-}1+\delta+2\delta^2\big) \bigr) +\bigl( \alpha_3^2-(\delta-1+c_0(1-2\delta))^2 \bigr)^2 \bigr\}.\label{A_at_simple_pole+_p6_coeffs}
\end{gather}

\subsection[Parameterization of $A(\lambda)$ at the pole of $y(x)$ as $\delta\neq0,\frac{1}{2},1$, and $\sigma=-1$]{Parameterization of $\boldsymbol{A(\lambda)}$ at the pole of $\boldsymbol{y(x)}$ as $\boldsymbol{\delta\neq0,\frac{1}{2},1}$, and $\boldsymbol{\sigma=-1}$}\label{A2}

The constant parameters in (\ref{A_regularized_p_at_simple_p6}) are determined by the parameters of the Laurent expansion as follows
\begin{gather*}
\hat a_3=-\delta+1,\qquad \hat b_3=a-1+c_0(2\delta-1), \nonumber\\
\hat c_3=\frac{1}{2(\delta-1)} \bigl\{ a^2\bigl[(\delta-2)^2-\alpha_2^2\bigr] +a\bigl[{-}\alpha_1^2+\alpha_3^2+\alpha_2^2
+(\delta-1)^2-2 -c_0\big(4(\delta-1)^2-2\delta\big) \bigr]\nonumber\\
\hphantom{\hat c_3=}{} -\alpha_3^2+(\delta-c_0(2\delta-1))^2 \bigr\},\nonumber\\
\hat b_+=\frac{\kappa_0}{a(a-1)(\delta-1)(2\delta-1)} \bigl\{ a^3\big({-}\alpha_2^2+(\delta-2)^2\big)+a^2\bigl[
\alpha_2^2-2+c_0\big(\alpha_2^2-8\big)\nonumber\\
\hphantom{\hat b_+=}{} +(\delta-2)\big(\alpha_1^2-\alpha_3^2\big) +\delta\big(\alpha_2^2-5\big) +2c_0\delta\big(11-\alpha_2^2\big) +c_0\delta^2(2\delta-13) -\delta^2(\delta-5) \bigr]\nonumber\\
\hphantom{\hat b_+=}{} -(c_0(2\delta-1)-1)\big(\alpha_3^2-(c_0+\delta-2c_0\delta)^2\big)
-a\bigl[ c_0^2(1-2\delta)^2(2\delta-5)\nonumber\\
\hphantom{\hat b_+=}{} -\delta\big(3+\alpha_1^2-\alpha_2^2+\delta-\delta^2\big) +c_0(2\delta-1)
\big(3+\alpha_1^2-\alpha_2^2+6\delta-3\delta^2\big)\nonumber\\
\hphantom{\hat b_+=}{} +\alpha_3^2(3-\delta-c_0(2\delta-1)) \bigr] \bigr\},\nonumber\\
\hat c_+=\frac{\kappa_0}{4a(a-1)(\delta-1)^2(2\delta-1)} \bigl\{ a^4\big(\alpha_2^2-(\delta-2)^2\big)^2\nonumber\\
\hphantom{\hat c_+=}{} -2a^3\big(\alpha_2^2-(\delta-2)^2\big) \bigl[ {-}1-2\delta+\delta^2 -\alpha_1^2+\alpha_3^2+\alpha_2^2 +c_0\big({-}4+10\delta-4\delta^2\big) \bigr]\nonumber\\
\hphantom{\hat c_+=}{} +a^2\bigl[ \alpha_1^4+\alpha_3^4+\big(\alpha_2^2-1\big)^2 +8c_0\big(1-\alpha_2^2\big) +c_0^2\big(24-2\alpha_2^2\big)+4\delta\big(1-\alpha_2^2\big)\nonumber\\
\hphantom{\hat c_+=}{} +c_0\delta\big(12+16\alpha_2^2\big) +c_0^2\delta\big({-}120+8\alpha_2^2\big)
+\delta^2\bigl(10 -88c_0 +c_0^2\big(198-8\alpha_2^2\big) \bigr)\nonumber\\
\hphantom{\hat c_+=}{}
+\delta^3\bigl({-}12 +72c_0 -120c_0^2 \bigr) +\delta^4\big(3-16c_0+24c_0^2\big) -2\alpha_1^2\big(1+\alpha_3^2+\alpha_2^2 -4c_0\nonumber\\
\hphantom{\hat c_+=}{} -6\delta+10c_0\delta+3\delta^2 - 4c_0\delta^2\big) +2\alpha_3^2\big({-}5+2\alpha_2^2+2\delta -2c_0\big(2-5\delta+2\delta^2\big)\big) \bigr]\nonumber\\
\hphantom{\hat c_+=}{} +2a\big(\alpha_3^2-(c_0+\delta-2c_0\delta)^2\big) \bigl[ 1+2\delta-\delta^2 +\alpha_1^2-\alpha_3^2-\alpha_2^2
+c_0\big(4-10\delta+4\delta^2\big) \bigr]\nonumber\\
\hphantom{\hat c_+=}{} +\big(\alpha_3^2-(c_0+\delta-2c_0\delta)^2\big)^2 \bigr\},\nonumber\\
\hat c_-=-\frac{a(a-1)(2\delta-1)}{\kappa_0}.
\end{gather*}

\subsection[Parameterization of $A(\lambda)$ at the pole of $y(x)$ as $\delta=1$ and $\sigma=-1$]{Parameterization of $\boldsymbol{A(\lambda)}$ at the pole of $\boldsymbol{y(x)}$ as $\boldsymbol{\delta=1}$ and $\boldsymbol{\sigma=-1}$} \label{A3}

The parameters in the coefficient matrix (\ref{A_regularized_p_at_simple_delta=1_p6}) are given by
\begin{gather*}
\check a_3=-\frac{1}{2}=-\delta+\frac{1}{2},\qquad \check b_3=c_0+\frac{3}{2}a-\frac{1}{2},\qquad \check c_3=a\left(\alpha_1-\frac{1}{2}\right),\\
\check b_+=\frac{\kappa_0}{a(a-1)}\bigl(a^2\big(1-\alpha_2^2\big) -a\big(2+\alpha_1^2-\alpha_3^2-\alpha_2^2-2c_0\big) -\alpha_3^2 +(c_0-1)^2
\bigr),\\
\check c_+=\kappa_0\left({-}\alpha_3^2 +\alpha_2^2 +\frac{a+1}{a-1}\alpha_1^2 +2\alpha_1 \left(\frac{c_0}{a-1}+1\right)\right),\qquad
\check b_-=-\frac{a(a-1)}{\kappa_0}.
\end{gather*}

\subsection[Parameterization of $A(\lambda)$ at the pole of $y(x)$ as $\delta=\frac{1}{2}$]{Parameterization of $\boldsymbol{A(\lambda)}$ at the pole of $\boldsymbol{y(x)}$ as $\boldsymbol{\delta=\frac{1}{2}}$} \label{A4}

The coefficients of $A(\lambda)$ (\ref{A_regularized_double_p6}) in terms of the monodromy exponents, the pole position $a\neq0,1$ and the arbitrary coefficient $c_{-2}\neq0$ of the Laurent expansion~(\ref{double_pole_p6}) are as follows
\begin{gather}
\tilde a_3=-\frac{3}{2}=-\delta-1,\qquad \tilde b_3=a+1-\frac{a^2(a-1)^2}{2c_{-2}},\nonumber\\
\tilde c_3= \frac{a^4(a-1)^4}{12c_{-2}^2} +\frac{a^2(a-1)^2(a+1)}{6c_{-2}} \nonumber\\
\hphantom{\tilde c_3=}{} +\frac{1}{12}\bigr(a^2\big(1-4\alpha_2^2\big) +a\big({-}7-4\alpha_1^2+4\alpha_3^2+4\alpha_2^2\big) +1-4\alpha_3^2
\bigr),\nonumber\\
\tilde c_+=-\frac{\kappa_0c_{-2}}{a^2(a-1)^2}, \nonumber\\
\tilde b_-= \frac{a^8(a-1)^8}{12\kappa_0c_{-2}^4} -\frac{(a-1)^4a^4}{12\kappa_0c_{-2}^2} \bigl(3\big(a^2-a+1\big) +4a\alpha_1^2 -4(a-1)\alpha_3^2
+4a(a-1)\alpha_2^2 \bigr)\nonumber\\
\hphantom{\tilde b_-=}{} +\frac{(a-1)^2a^2}{12\kappa_0c_{-2}} \bigl( -2a^3+3a^2+3a-2 -4a(a+1)\alpha_1^2 +4\big(a^2-3a+2\big)\alpha_3^2 \nonumber\\
\hphantom{\tilde b_-=}{} +4a\big(2a^2-3a+1\big)\alpha_2^2\bigr),\nonumber\\
\tilde c_-= \frac{a^{10}(a-1)^{10}}{144\kappa_0c_{-2}^5} +\frac{a^8(a-1)^8(a+1)}{36\kappa_0c_{-2}^4}\nonumber\\
\hphantom{\tilde c_-=}{} -\frac{a^6(a-1)^6}{72\kappa_0c_{-2}^3} \bigl({-}3\big(a^2-a+1\big) +4a\alpha_1^2 -4(a-1)\alpha_3^2 +4a(a-1)\alpha_2^2
\bigr) \nonumber\\
\hphantom{\tilde c_-=}{} -\frac{a^4(a-1)^4(a+1)}{36\kappa_0c_{-2}^2} \bigl({-}a^2+7a-1 +4a\alpha_1^2 -4(a-1)\alpha_3^2 +4a(a-1)\alpha_2^2\bigr) \nonumber\\
\hphantom{\tilde c_-=}{} +\frac{a^2(a-1)^2}{144\kappa_0c_{-2}} \bigl( \big(1-4\alpha_3^2\big)^2 +a^4\big(1-4\alpha_2^2\big)^2 +2a\big({-}1+4\alpha_3^2\big)\big(7+4\alpha_1^2-4\alpha_3^2 -4\alpha_2^2\big)\nonumber\\
\hphantom{\tilde c_-=}{} +2a^3\big(7+4\alpha_1^2-4\alpha_3^2-4\alpha_2^2\big) \big({-}1+4\alpha_2^2\big) +a^2\big(51+16\alpha_1^4+16\alpha_3^4-64\alpha_2^2 +16\alpha_2^4\nonumber\\
\hphantom{\tilde c_-=}{} +64\alpha_3^2\big({-}1+\alpha_2^2\big) -8\alpha_1^2\big(11+4\alpha_3^2+4\alpha_2^2\big)\big)\bigr).\label{hatA_coeffs_double_p6}
\end{gather}

\subsection[Asymptotic coefficients of $\Psi_{\infty}(\lambda)$ as $\lambda\to\infty$]{Asymptotic coefficients of $\boldsymbol{\Psi_{\infty}(\lambda)}$ as $\boldsymbol{\lambda\to\infty}$}\label{psi_k_p6}

The asymptotic parameters in (\ref{Psi_p6_as_at_8}) are expressed as follows
\begin{gather}
\delta\neq\pm\frac{1}{2},\pm1\colon \quad \psi_1=\frac{\kappa}{2\delta-1}\sigma_+-\frac{\tilde\kappa}{2\delta+1}\sigma_-,\qquad
d_1=-p-\delta(2y-x-1),\nonumber\\
\psi_2=\kappa\frac{2p+(2\delta+1)y-x-1}{4(\delta-1)\big(\delta-\frac{1}{2}\big)}\sigma_+-\tilde\kappa\frac{2p-(2\delta+1)\tilde y+x+1+4\delta y}{4(\delta+1)\big(\delta+\frac{1}{2}\big)}\sigma_-,\nonumber\\
d_{21}=\frac{\kappa\tilde\kappa}{2(1+2\delta)}+\frac{1}{2}\bigl(p(y-x-1)+\delta\big((y-x-1)^2-x\big)-z\bigr),\nonumber\\
d_{22}=\frac{\kappa\tilde\kappa}{2(1-2\delta)} -\frac{1}{2}\bigl(p(y-x-1)+\delta\big((y-x-1)^2-x\big)-z\bigr).\label{deltaneq1/2_1_Psi_p6_as_at_8}
\end{gather}

\subsection*{Acknowledgments}
A.K.~was supported by the project SPbGU 11.38.215.2014. He also thanks the of SISSA where this project was originated. Many thanks to the anonymous referees for their suggestions towards improving the manuscript.

\pdfbookmark[1]{References}{ref}
\LastPageEnding

\end{document}